\documentclass[12pt]{amsart}
\usepackage{}
\usepackage{amsmath}
\usepackage{amsfonts}
\usepackage{amssymb}
\usepackage[all,cmtip]{xy}           

\usepackage{bbm}
\usepackage{bbding}
\usepackage{txfonts}
\usepackage[shortlabels]{enumitem}
\usepackage{ifpdf}
\ifpdf
\usepackage[colorlinks,final,backref=page,hyperindex]{hyperref}
\else
\usepackage[colorlinks,final,backref=page,hyperindex,hypertex]{hyperref}
\fi
\usepackage{amscd}
\usepackage{tikz-cd}
\usepackage[active]{srcltx}

\makeatletter

\newtheorem{theorem}{Theorem}[section]
\newtheorem{prop}[theorem]{Proposition}
\newtheorem{lemma}[theorem]{Lemma}
\newtheorem{coro}[theorem]{Corollary}
\newtheorem{prop-def}{Proposition-Definition}[section]

\theoremstyle{definition}
\newtheorem{defn}[theorem]{Definition}

\newcommand{\nc}{\newcommand}

\nc{\delete}[1]{{}}
\nc{\mmargin}[1]{}
\nc{\Alg}{\mathrm{Alg}}
\nc{\AvO}{\mathrm{AvO}}
\nc{\AvA}{\mathrm{AvA}}
\nc{\rmH}{\mathrm{H}}
\nc{\mlabel}[1]{\label{#1}}  
\nc{\mcite}[1]{\cite{#1}}  
\nc{\mref}[1]{\ref{#1}}  
\nc{\mbibitem}[1]{\bibitem{#1}} 

\delete{
	\nc{\mlabel}[1]{\label{#1}  
		{\hfill \hspace{1cm}{\bf{{\ }\hfill(#1)}}}}
	\nc{\mcite}[1]{\cite{#1}{{\bf{{\ }(#1)}}}}  
	\nc{\mref}[1]{\ref{#1}{{\bf{{\ }(#1)}}}}  
	\nc{\mbibitem}[1]{\bibitem[\bf #1]{#1}} 
}

 \font\cyrs=wncyr7

\newcommand{\bk}{{\mathbf{k}}}

\nc{\vep}{\varepsilon}
\nc{\bin}[2]{ (_{\stackrel{\scs{#1}}{\scs{#2}}})}  
\nc{\binc}[2]{(\!\! \begin{array}{c} \scs{#1}\\
		\scs{#2} \end{array}\!\!)}  
\nc{\bincc}[2]{  ( {\scs{#1} \atop
		\vspace{-1cm}\scs{#2}} )}  
\nc{\oline}[1]{\overline{#1}}
\nc{\mapm}[1]{\lfloor\!|{#1}|\!\rfloor}
\nc{\bs}{\bar{S}}
\nc{\la}{\longrightarrow}
\nc{\ot}{\otimes}
\nc{\rar}{\rightarrow}
\nc{\lon }{\,\rightarrow\,}
\nc{\dar}{\downarrow}
\nc{\dap}[1]{\downarrow \rlap{$\scriptstyle{#1}$}}
\nc{\defeq}{\stackrel{\rm def}{=}}
\nc{\dis}[1]{\displaystyle{#1}}
\nc{\dotcup}{\ \displaystyle{\bigcup^\bullet}\ }
\nc{\hcm}{\ \hat{,}\ }
\nc{\hts}{\hat{\otimes}}
\nc{\hcirc}{\hat{\circ}}
\nc{\lleft}{[}
\nc{\lright}{]}
\nc{\curlyl}{\left \{ \begin{array}{c} {} \\ {} \end{array}
	\right .  \!\!\!\!\!\!\!}
\nc{\curlyr}{ \!\!\!\!\!\!\!
	\left . \begin{array}{c} {} \\ {} \end{array}
	\right \} }
\nc{\longmid}{\left | \begin{array}{c} {} \\ {} \end{array}
	\right . \!\!\!\!\!\!\!}
\nc{\ora}[1]{\stackrel{#1}{\rar}}
\nc{\ola}[1]{\stackrel{#1}{\la}}
\nc{\scs}[1]{\scriptstyle{#1}} \nc{\mrm}[1]{{\rm #1}}
\nc{\dirlim}{\displaystyle{\lim_{\longrightarrow}}\,}
\nc{\invlim}{\displaystyle{\lim_{\longleftarrow}}\,}
\nc{\dislim}[1]{\displaystyle{\lim_{#1}}} \nc{\colim}{\mrm{colim}}
\nc{\mvp}{\vspace{0.3cm}} \nc{\tk}{^{(k)}} \nc{\tp}{^\prime}
\nc{\ttp}{^{\prime\prime}} \nc{\svp}{\vspace{2cm}}
\nc{\vp}{\vspace{8cm}}
\nc{\modg}[1]{\!<\!\!{#1}\!\!>}
\nc{\intg}[1]{F_C(#1)}
\nc{\lmodg}{\!<\!\!}
\nc{\rmodg}{\!\!>\!}
\nc{\cpi}{\widehat{\Pi}}
\nc{\ssha}{{\mbox{\cyrs X}}} 
\nc{\tsha}{{\mbox{\cyrt X}}}
\nc{\shpr}{\diamond}    
\nc{\labs}{\mid\!}
\nc{\rabs}{\!\mid}

\nc{\ad}{\mrm{ad}}
\nc{\ann}{\mrm{ann}}
\nc{\Aut}{\mrm{Aut}}
\nc{\bim}{\mbox{-}\mathsf{Bimod}}
\nc{\br}{\mrm{bre}}
\nc{\can}{\mrm{can}}
\nc{\Cont}{\mrm{Cont}}
\nc{\rchar}{\mrm{char}}
\nc{\cok}{\mrm{coker}}
\nc{\de}{\mrm{dep}}
\nc{\dtf}{{R-{\rm tf}}}
\nc{\dtor}{{R-{\rm tor}}}

\nc{\Div}{{\mrm Div}}
\nc{\Diff}{\mrm{DA}}
\nc{\Diffl}{\mathsf{DA}_\lambda}
\nc{\diffo}{{\mathsf{DO}_\lambda}}
\nc{\alg}{\mathsf{Alg}}
\nc{\End}{\mrm{End}}
\nc{\Ext}{\mrm{Ext}}
\nc{\Fil}{\mrm{Fil}}
\nc{\Fr}{\mrm{Fr}}
\nc{\Frob}{\mrm{Frob}}
\nc{\Gal}{\mrm{Gal}}
\nc{\GL}{\mrm{GL}}
\nc{\Hom}{\mrm{Hom}}
\nc{\Hoch}{\mrm{Hoch}}
\nc{\hsr}{\mrm{H}}
\nc{\hpol}{\mrm{HP}}
\nc{\id}{\mrm{id}}
\nc{\im}{\mrm{im}}
\nc{\Id}{\mrm{Id}}
\nc{\ID}{\mrm{ID}}
\nc{\Irr}{\mrm{Irr}}
\nc{\incl}{\mrm{incl}}
\nc{\length}{\mrm{length}}
\nc{\NLSW}{\mrm{NLSW}}
\nc{\Lie}{\mrm{Lie}}
\nc{\mchar}{\rm char}
\nc{\mpart}{\mrm{part}}
\nc{\ql}{{\QQ_\ell}}
\nc{\qp}{{\QQ_p}}
\nc{\rank}{\mrm{rank}}
\nc{\rcot}{\mrm{cot}}
\nc{\rdef}{\mrm{def}}
\nc{\rdiv}{{\rm div}}
\nc{\rtf}{{\rm tf}}
\nc{\rtor}{{\rm tor}}
\nc{\res}{\mrm{res}}
\nc{\SL}{\mrm{SL}}
\nc{\Spec}{\mrm{Spec}}
\nc{\tor}{\mrm{tor}}
\nc{\Tr}{\mrm{Tr}}
\nc{\tr}{\mrm{tr}}
\nc{\wt}{\mrm{wt}}
\def\ot{\otimes}

\nc{\bfk}{{\bf k}}
\nc{\bfone}{{\bf 1}}
\nc{\bfzero}{{\bf 0}}
\nc{\detail}{\marginpar{\bf More detail}
	\noindent{\bf Need more detail!}
	\svp}
\nc{\gap}{\marginpar{\bf Incomplete}\noindent{\bf Incomplete!!}
	\svp}
\nc{\FMod}{\mathbf{FMod}}
\nc{\Int}{\mathbf{Int}}
\nc{\Mon}{\mathbf{Mon}}
\nc{\remarks}{\noindent{\bf Remarks: }}
\nc{\Rep}{\mathbf{Rep}}
\nc{\Rings}{\mathbf{Rings}}
\nc{\Sets}{\mathbf{Sets}}
\nc{\ob}{\mathsf{Ob}}

\nc{\BA}{{\mathbb A}}   \nc{\CC}{{\mathbb C}}
\nc{\DD}{{\mathbb D}}   \nc{\EE}{{\mathbb E}}
\nc{\FF}{{\mathbb F}}   \nc{\GG}{{\mathbb G}}
\nc{\HH}{{\mathbb H}}   \nc{\LL}{{\mathbb L}}
\nc{\NN}{{\mathbb N}}   \nc{\PP}{{\mathbb P}}
\nc{\QQ}{{\mathbb Q}}   \nc{\RR}{{\mathbb R}}
\nc{\TT}{{\mathbb T}}   \nc{\VV}{{\mathbb V}}
\nc{\ZZ}{{\mathbb Z}}   \nc{\TP}{\widetilde{P}}
\nc{\m}{{\mathbbm m}}

\nc{\cala}{{\mathcal A}}    \nc{\calc}{{\mathcal C}}
\nc{\cald}{\mathcal{D}}     \nc{\cale}{{\mathcal E}}
\nc{\calf}{{\mathcal F}}    \nc{\calg}{{\mathcal G}}
\nc{\calh}{{\mathcal H}}    \nc{\cali}{{\mathcal I}}
\nc{\call}{{\mathcal L}}    \nc{\calm}{{\mathcal M}}
\nc{\caln}{{\mathcal N}}    \nc{\calo}{{\mathcal O}}
\nc{\calp}{{\mathcal P}}    \nc{\calr}{{\mathcal R}}
\nc{\cals}{{\mathcal S}}    \nc{\calt}{{\Omega}}
\nc{\calv}{{\mathcal V}}    \nc{\calw}{{\mathcal W}}
\nc{\calx}{{\mathcal X}}

\nc{\fraka}{{\mathfrak a}}
\nc{\frakb}{\mathfrak{b}}
\nc{\frakg}{{\frak g}}
\nc{\frakl}{{\frak l}}
\nc{\fraks}{{\frak s}}
\nc{\frakB}{{\frak B}}
\nc{\frakm}{{\frak m}}
\nc{\frakM}{{\frak M}}
\nc{\frakp}{{\frak p}}
\nc{\frakW}{{\frak W}}
\nc{\frakX}{{\frak X}}
\nc{\frakS}{{\frak S}}
\nc{\frakA}{{\frak A}}
\nc{\frakx}{{\frakx}}

\nc{\lir}[1]{\textcolor{red}{\underline{Li:}#1 }}

\begin{document}

\title[Homotopy averaging algebras]{Cohomology theory  of averaging  algebras, $L_\infty$-structures  and homotopy averaging algebras}

\author{Kai Wang and Guodong Zhou}
\address{  School of Mathematical Sciences, Shanghai Key Laboratory of PMMP,
  East China Normal University,
 Shanghai 200241,
   China}
\email{wangkai@math.ecnu.edu.cn }
\address{  School of Mathematical Sciences, Shanghai Key Laboratory of PMMP,
  East China Normal University,
 Shanghai 200241,
   China}
 
 \email{gdzhou@math.ecnu.edu.cn}

\date{\today}

\begin{abstract}This paper studies averaging algebras, say,  associative algebras endowed with averaging operators. We develop a cohomology theory for averaging algebras and   justify it by interpreting lower degree cohomology groups as   formal deformations and  abelian extensions of averaging  algebras. We make explicit the $L_\infty$-algebra structure over the   cochain complex defining   cohomology groups and  introduce  the notion of homotopy averaging algebras as   Maurer-Cartan elements of this $L_\infty$-algebra.

\end{abstract}

\subjclass[2020]{
16E40   
16S80   
16S70   
17B38 
}

\keywords{ averaging algebra, averaging operator,  cohomology, extension, deformation,  homotopy averaging algebra,  $L_\infty$-algebra}

\maketitle

\tableofcontents

\allowdisplaybreaks

\section{Introduction}

Throughout this paper, $\bk$ denotes a field. All the vector spaces and algebras are over $\bk$ and all tensor products are also taking over $\bk$.

Let $R$ be an associative algebra over field $\bk$. An averaging operator over $R$ is a $\bk$-linear map $A: R\to R$ such that $$A(x)A(y)=A(A(x)y)=A(xA(y)$$ for all $x, y\in R$.

 The research on  averaging operators has a long history and it appears in various mathematical branches from turbulence theory to functional analysis, probability theory etc.
   When investigating  turbulence theory,   Reynolds already studied implicitly the averaging operator  in a famous
paper \cite{Rey}  published in 1895.
In a series of papers of 1930s, Kolmogoroff and Kamp\'{e} de F\'eriet introduced  explicitly the averaging operator
 in the context of turbulence theory and functional analysis \cite{KdF}\cite{Miller}. Birkhoff continued
the line of research  in \cite{Bir}.
Moy investigated   averaging operators from the viewpoint of  conditional expectation
in probability theory \cite{Moy}.
Kelley  \cite{Kelley} and Rota \cite{Rota}  studied the role of averaging operators in
 Banach algebras.

 Contrary to the above studies of analytic nature, the algebraic study on averaging
operators began with the Ph.D thesis of  Cao \cite{Cao}. He  constructed explicitly    free unitary commutative averaging algebras and discovered  the
Lie algebra structures induced naturally from averaging operators.
Inspired by the work of  Loday \cite{Loday} who
defined a dialgebra as the enveloping algebra of a Leibniz algebra,
Aguiar associated   a dissociative algebra  to an averaging associative
algebra \cite{Agu}.
In a paper \cite{PeiGuo} Guo and Pei  studied averaging operators from an algebraic and combinatorial
point of view. They first construct free nonunital averaging algebras in terms of a class
of bracketed words called averaging words, related them to  large Schr?der numbers and
 unreduced trees from an  operadic point of view. Another paper by Gao and Zhang \cite{GaoZhang}
 contains an explicit construction of  free  unital averaging algebras  in terms of  bracketed polynomials and the main tools were rewriting systems and Gr\"{o}bner-Shirshov bases.

The averaging operators attract much attention also partly     because of their   closely relationship to
other operators such as Reynolds operators, symmetric operators and Rota-Baxter operators \cite{Bong}\cite{GamMil}\cite{Tri},
 the latter having many applications in many other fields of mathematics  \cite{Bax}\cite{Guo}.

\bigskip

In the paper we study averaging operators from the viewpoint of deformation theory. We construct a cohomology theory which controls simultaneous deformations of associative multiplication and the averaging operator.
Because of the complexity of the problem, the cochain complex is no longer a differential graded Lie algebra (DGLA) but an $L_\infty$-algebra. We  give  explicitly the higher Lie brackets over the cochain complex. As a byproduct, we define homotopy averaging algebras as   Maurer-Cartan elements of this $L_\infty$-algebra.

We would like to mention several papers which are somehow parallel to this paper. Motivated from geometry and physics, Sheng, Tang and Zhu \cite{STZ} studies embedding tensors (another name of averaging operators in physics) for Lie algebras and they construct a cohomology theory for such operators on Lie algebras using derived brackets as a main tool. Another work of Chen, Ge and Xiang \cite{CGX} studies  embedding tensors using operadic tool and they compute explicitly the Boardmann-Vogt resolution of the colored operad governing  embedding tensors.  We warmly recommend the comparative reading of their papers.

\bigskip

The layout of this paper is as follows: The first section contains basic definitions and facts about averaging algebras and bimodules over them. We define the cohomology theory of averaging operators and averaging algebras in Section 2. Deformation theory of averaging algebras is developed in Section 4 and abelien extensions of averaging algebras are interpreted as the second cohomology group in Section 5. We construct explicitly the $L_\infty$-algebra structure over the cochain complex computing cohomology theory of   averaging algebras in Section 6 and in the last section we introduce the notion of homotopy averaging algebras as Maurer-Cartan elements of this $L_\infty$-algebra.

\bigskip

\section{Averaging algebras and their bimodule} .

In this section, we introduce averaging algebras and bimodules over them and present some basic observations.

\begin{defn}
	Let $(R,\cdot)$ be an associative algebra over field $\bk$. If is endowed with a linear operator $A: R\rightarrow R$ such that
	\begin{equation}\label{avg}
		A(x)A(y)=A(xA(y))=A(A(x)y), \forall x,y \in R,
	\end{equation}
	then we call $(R,\cdot,A)$ an averaging algebra.
	
\end{defn}

 Given two averaging algebras $(R,\cdot, A)$ and $(R',\cdot', A')$, a morphism of averaging algebras from $(R,\cdot, A)$ to $(R',\cdot', A')$ is a homomorphism of algebras $\phi: (R,\cdot)\rightarrow (S,\cdot')$ satisfying $\phi\circ A=A'\circ \phi.$

 \medskip

 \begin{defn}
 	Let $(R,\cdot, A)$ be an averaging algebra.
 	\begin{itemize}
 		\item[(i)] A bimodule over the averaging algebra $(R,\cdot,A)$ is a bimodule $M$ over associative algebra $(R,\cdot)$ endowed with an operator $A_M: M\rightarrow M$, such that for any $r\in R,m\in M$, the following equalities hold:
 		\begin{eqnarray}A(r)A_M(m)=&A_M(A(a)m)&=A_M(aA_M(m)),\\
 		A_M(m)A(r)=&A_M(A_M(m)a)&=A_M(mA(a)).
 		\end{eqnarray}
 		\item[(ii)] Given two bimodules $(M,A_M)$ and $(N,A_N)$ over averaging algebra $(R,\cdot,A)$, a morphism from $(M,A_M)$ to $(N,A_N)$ is a bimodule morphism $f: M\rightarrow N$ over $(R,\cdot)$ such that :
 		\[f\circ A_M=A_N\circ f.\]
 	\end{itemize}
 \end{defn}
 The algebra $R$ itself is naturally a bimodule over $(R,\cdot,A)$ with $A_R=A$, called the regular bimodule.
\medskip

The following easy result is left as an exercise.
\begin{prop}
Let $(M,A_M)$ be a bimodule over averaging algebra $(R,\cdot,A)$. Then $R\oplus M$ is an averaging algebra, where the averaging operator is $A\oplus A_M$ and the multiplication is:
\[(a, m)\cdot (b, n)=(ab, an+mb),\]
for all $a, b\in R, m, n\in M$. 	
\end{prop}

\medskip

\begin{prop}\label{newmultiplication}Let $(R,\cdot,A)$ be an averaging algebra. Then the following operations:
	\begin{align*}a\star b&:=aA(b),\\
	a\diamond b&:= A(a)b
	\end{align*}
	are both associative.
	
	Conversely, let  $(R,\cdot)$ be a unital algebra over $\bk$ endowed
	with a linear operator $A$ such that operations $\star,\diamond$
	are associative. Then $(R,\cdot,A)$ is  an averaging algebra.
\end{prop}
\begin{proof}By definition, we have identities for any $a,b, c\in R$:
	\begin{align*}
	(a\star b)\star c&=(a\cdot A(b))\star c\\
	&=a\cdot A(b)A(c)\\
	&=a\cdot(A(b\cdot Ac))\\
	&= a\star (b\star c).
	\end{align*}
	 Thus the operation $\star$ is associative. Analogously, one can check the associativity of operation $\diamond$.

	Conversely, assume that unital algebra $(R,\cdot)$ is endowed with a linear map $A: R\rightarrow R$ such that operations $\star,\diamond$ is associative. Then
	\begin{align*}
	A(a)A(b)&=(1_R\cdot A(a))A(b)\\
	&=(1_R\star a)\star b\\
	&=1_R\star(a\star b)\\
	&=1_R\star(a\cdot A(b))\\
	&=1_R\cdot A(a\cdot A(b))\\
	&=A(a\cdot (A(b))).
	\end{align*}
hold for any $a,b\in R$. Similarly, equality $A(a)A(b)=A(A(a)b)$ also holds. Thus $(R,\cdot,A)$ is an averaging algebra.
	\end{proof}

\medskip
\begin{prop}
	Let $(R,\cdot, A)$  be an averaging algebra. Then
	$(R,\star,A),(R,\diamond,A)$ are also averaging algebras.	
\end{prop}
\begin{proof}
	For any $a,b\in R$, we have
	\begin{align*}
	A(a)\star A(b)&=A(a)\cdot(A(A(b)))\\
	&=A(a\cdot A^2(b))\\
	&=A(A(a)\cdot A(b)),	
	\end{align*}
	and
	\begin{align*}
	A(a\star A(b))&=A(a\cdot A^2(b)),\\
	A(A(a)\star b)&=A(A(a)\cdot A(b)).
	\end{align*}
	Thus the relations $A(a)\star A(b)=A(a\star A(b))=A(A(a)\star b)$ hold. So $(R,\star, A)$ is an averaging algebra.
	It  is analogous to prove that $(R,\diamond, A)$ is an averaging algebra.
	\end{proof}

\bigskip

Let $(R,\cdot, A)$ be an averaging algebra and $(M,A_M)$ be a bimodule over $(R,\cdot, A)$. Then
we can make $M$ into a bimodule over $(R,\star) $ and $(R,\diamond)$. For any $a\in R$ and $m\in M$, we define the following actions:
\begin{align*}
a\vdash  m:=A(a)m-A_M(am),\  m\dashv a:=mA(a)
\end{align*}
and
\begin{align*}
a \rhd m:=A(a)m,\ m\lhd a:=m A(a)-A_M(mb).
\end{align*}

\begin{prop}\label{newbimodule}
	
	The action $(\vdash, \dashv)$ (resp. $(\rhd, \lhd)$) makes $M$ become a bimodule over associative algebra $(R,\star)$ (resp. $(R,\diamond)$).
\end{prop}
\begin{proof}
	For any $a,b\in R$ and $m\in M$, we have
	\begin{align*}
	&a\vdash(b\vdash m)\\
	&=a\vdash(A(b)m-A_M(bm))\\
	&=A(a)\cdot(A(b)m-A_M(bm))-A_M(a\cdot A(b)\cdot m-a\cdot A_M(b\cdot m))\\
	&=A(a)\cdot A(b)\cdot m-A_M(a\cdot A(b)\cdot m),
	\\\\
	&(a\star b)\vdash m\\
	&=(a\cdot A(b))\vdash m\\
	&= A(a\cdot A(b))\cdot m-A_M(a\cdot A(b)\cdot m)\\
	&=a\vdash(b\vdash m).
	\end{align*}
	 Thus operation $\vdash$ makes $M$ into a left module over $(R, \star)$. Analogously, operation $\dashv$ makes $M$ into a right module over $(R,\star)$.
	
	 Moreover, we have :
	 \begin{align*}
	 &a\vdash(m \dashv b)\\
	 &=a\vdash(m\cdot A(b))\\
	 &=A(a)\cdot m\cdot A(b)- A_M(a\cdot m\cdot A(b))\\
	 &=A(a)\cdot m\cdot A(b)-A_M(a\cdot m)\cdot A(b)\\
	 &=(A(a)\cdot m-A_M(a\cdot m))\cdot A(b)\\
	 &=(a\vdash m)\cdot A(b)\\
	 &=(a\vdash m)\dashv b
	 \end{align*}
	
	Thus operations $(\vdash,  \dashv)$ makes $M$ into a bimodule over associative algebra $(R,\star)$.
	
Similarly, one can check that operations $(\rhd, \lhd)$ makes $M$ into a bimodule over $(R,\diamond)$.
	
	\end{proof}

\bigskip

\section{Cohomology theory of averaging algebras} \label{sec:cohomologyava}

Let $M$ be a bimodule over an associative algebra R. Recall that the Hochschild cohomology of $R$ with coefficients in $M$:  $$(C_{\Alg}^\bullet(A,M)=\bigoplus\limits_{n=0}^\infty C_{\Alg}^n(R,M), \delta),$$ where $C_{\Alg}^n(R,M)=(\Hom(R^{\ot n},M)$ and the differential
$\delta: C^n_{\Alg}(R,M)\rightarrow C^{n+1}_{\Alg}(R,M)$  is given by

\begin{align*}\delta(f)(x_1\ot\dots\ot x_{n+1})=&x_1f(x_2\ot \dots\ot x_n)+\sum_{i=1}^n(-1)^if(x_1\ot \dots\ot x_ix_{i+1}\ot \dots\ot x_{n+1})\\
& +(-1)^{n+1}f(x_1\ot\dots\ot x_{n})x_{n+1}
\end{align*}
for all $f\in C^n(R,M), x_1,\dots,x_{n+1}\in R$. The corresponding Hochschild cohomology is denoted $\mathrm{HH}^\bullet(R,M)$. When $M=R$, just denote the Hochschild cochain complex with coefficients in $R$ by $C^\bullet_{\Alg}(R)$ and denote the Hochschild cohomology by $\mathrm{HH}^\bullet(R)$.

\bigskip

\subsection{Cohomology of averaging operators}\label{sec:cohomologyao}
Let $(R,\cdot, A)$ be an averaging algebra and $(M,A_M)$ be a bimodule over $(R,\cdot,A)$. In this subsection, we define the cohomology of averaging operators.
By Proposition~\ref{newmultiplication}, the averaging operator $A$ induces two new multiplications $\star$ and $\diamond$ on $R$, and by
Proposition~\ref{newbimodule}, operations $(\vdash, \dashv)$ (resp.
$(\rhd, \lhd)$) makes $M$ into a bimodule over $(R,\star)$ (resp.
$(R,\diamond)$).

 Consider the Hochschild cochain complex of $(R,\star)$ with the
 coefficient in bimodule $(M, \vdash, \dashv)$. Denote this cochain
 complex by $$C_r^\bullet(R,M)=\bigoplus\limits_{i=0}^\infty C_r^n(R,M),$$
 where $C^n_r(R,M)=\Hom(R^{\ot n},M)$ and its differential $ \partial_r: C^n_r(R,M)\rightarrow C^{n+1}_r(R,M)$
 is given by :
 \begin{align*}
 \partial_r(f)(x_1\ot \dots \ot x_{n+1})=&x_1\vdash  f(x_2\ot \dots \ot x_{n+1})+\sum_{i=1}^n(-1)^if(x_1\ot \dots \ot x_i\star x_{i+1}\ot \dots \ot x_{n+1})\\
 &+(-1)^{n+1}f(x_1\ot \dots x_n)\dashv x_{n+1}\\
 =&A(x_1)f(x_2\ot \dots \ot x_{n+1})-A_M\big( x_1f(x_2\ot \dots \ot x_{n+1})\big)\\
 &+\sum_{i=1}^n(-1)^if(x_1\ot \dots \ot x_iA(x_{i+1})\ot \dots x_{n+1})\\
 &+(-1)^{n+1}f(x_1\ot \dots \ot x_{n})A(x_{n+1}).
 \end{align*}
  for all $f\in C^n_r(R,M), x_1,\dots,x_{n+1}\in R$.

 For algebra $(R,\diamond)$, we denote its Hochschild cochain complex
 with coefficients in bimodule  $(M,\rhd, \lhd)$ by
 $$C_l^\bullet(R,M)=\bigoplus\limits_{i=0}^\infty C^n_l(R,M),$$ where $C^n_l(R,M)=\Hom(R^{\ot n},M)$ and by definition, its differential $\partial_l: C^n_l(R,M)\rightarrow C^{n+1}_l(R,M)$ is given by:
\begin{align*}
 \partial_l(g)(x_1\ot \dots\ot x_{n+1})=&x_1\rhd g(x_2\ot \dots \ot x_{n+1})+\sum_{i=1}^n(-1)^ig(x_1\ot\dots \ot x_i\diamond x_{i+1}\ot \dots \ot x_{n+1})\\
 &+(-1)^{n+1}g(x_1\ot \dots \ot x_n)\lhd x_{n+1}\\
 =&A(x_1)g(x_2\ot \dots \ot x_{n+1})+\sum_{i=1}^n(-1)^ig(x_1\ot \dots \ot A(x_i)x_{i+1}\ot \dots \ot x_{n+1})\\
& +(-1)^{n+1}g(x_1\ot \dots \ot x_n)A(x_{n+1})+(-1)^nA_M(g(x_1\ot \dots \ot x_n)x_{n+1})
 \end{align*}
 for all $g\in C_l^n(R,M)$, $x_1,\dots,x_n\in R$. When $M=R$, we just denote $C^\bullet_l(R,R)$ by $C^\bullet_l(R)$, and denote $C^\bullet_r(R,R)$ by $C^\bullet_r(R)$.

 Identifying $C^0_r(R,M)$ with $C^0_l(R,M)$, identifying $C^1_l(R,M)$ with $C^1_r(R,M)$ and taking the direct sum of $C^n_l(R,M)$ and $C^n_r(R,M)$ for $n\geqslant 2$,  we get the following complex:

 $$\xymatrixcolsep{3.5pc}\xymatrix{
 	 \Hom(k,M)\ar[r]^-{\partial_0}&
	\Hom(R,M)\ar[r]^-{\begin{pmatrix}\partial_r\\\partial_l\end{pmatrix}}& {\begin{matrix}C^2_r(R,M)\\ \bigoplus\\
 	C^2_l(R,M)\end{matrix}}\ar[r]^-{\begin{pmatrix}\partial_r&0\\0&\partial_l\end{pmatrix}}& {\cdots\cdots\begin{matrix}C^n_r(R,M)\\ \bigoplus\\
 	C^n_l(R,M)\end{matrix}}\ar[r]^-{\begin{pmatrix}\partial_r&0\\0&\partial_l\end{pmatrix}}&{\begin{matrix}C^{n+1}_r(R,M)\\ \bigoplus\\
 	C^{n+1}_l(R,M)\end{matrix}\cdots\cdots}
 },$$
where $\partial_0$ is defined as: for $f\in \Hom(k,M)$, assume that $f(1)=m$, then $$\partial_0(f)(r)=A_M(mr)-A_M(rm)-mA(r)+A(r)m$$ for any $r\in R$. One can check that $$\begin{pmatrix}\partial_r\\\partial_l\end{pmatrix}\circ \partial_0=0.$$

We denote the above complex by $C_{\AvO}^\bullet(R,M)$, called the \textit{cochain complex of the averaging operator $A$ with coefficients in bimodule $(M,A_M)$}. Its cohomology is denoted by  $\rmH^\bullet_{\AvO}(R,M)$, called the \textit{cohomology of the averaging operator $A$ of averaging algebra $(R,\cdot, A)$ with coefficients in $(M,A_M)$}. When $(M,A_M)=(R,A)$, we just denote $C_{\AvO}^\bullet(R, R)$ by $C_{\AvO}^\bullet(R)$ and call it the cochain complex of averaging operator. We denote $\rmH^\bullet_{\AvO}(R,M)$ by $\rmH^\bullet_{\AvO}(R)$, called the cohomology of averaging operator $A$.

\medskip

\subsection{Cohomology of averaging algebras}
\label{sec:cohomologyav}
In this subsection, we'll define the cohomology of averaging algebras and such a cohomology theory involves both the multiplication and the averaging operator  in an averaging algebra.

 Firstly, let's build a morphism of complexes: $\Phi: C_{\Alg}^\bullet(R,M)\rightarrow C^\bullet_{\AvO}(R,M) $,
 $$
 \xymatrixcolsep{3.5pc}\xymatrix{C^0_{\Alg}(R,M)\ar[d]^-{\Phi^0}\ar[r]^-{\delta^0}&{ C^1_{\Alg}(R,M)}\ar[d]^-{\Phi^1}\ar[r]^-{\delta^1}&C^2_{\Alg}(R,M)\ar[d]^-{\Phi^2}\ar@{.}[r]&{ C^n_{\Alg}(R,M)}\ar[d]^-{\Phi^n}\ar[r]^-{\delta^n}&C^{n+1}_{\Alg}(R,M)\ar[d]^-{\Phi^{n+1}}\ar@{.}[r]&\\
 	\Hom(k,M)\ar[r]^-{\partial^0}&{\Hom(R,M)}\ar[r]^-{\begin{pmatrix}\partial_r^1\\\partial_l^1\end{pmatrix}}&{\begin{matrix}C^2_r(R,M)\\ \bigoplus\\
 		C^2_l(R,M)\end{matrix}}\ar@{.}[r]&{\begin{matrix}C^n_r(R,M)\\ \bigoplus\\
 	C^n_l(R,M)\end{matrix}}\ar[r]^-{\begin{pmatrix}\partial_r^n&0\\0&\partial_l^n\end{pmatrix}}&{\begin{matrix}C^{n+1}_r(R,M)\\ \bigoplus\\
 C^{n+1}_l(R,M)\end{matrix}}\ar@{.}[r]&
}
$$
Here we define
\begin{align*}\Phi^0=\mathrm{id},\Phi^1(f)= f\circ A-A_M\circ f
\end{align*} for any $f\in \Hom(R,M)$,  and when $n\geqslant 2$, $\Phi^n=\begin{pmatrix}\Phi_r^n\\\Phi_l^n\end{pmatrix}$, where $\Phi_r^n: C^n(R)$ \begin{align*}
 \Phi_r^n(f)&=f\circ  A^{\ot n} -A_M\circ  f\circ (\mathrm{id}\ot A^{\ot n-1}) , \\
 \Phi_l^n(f)&=f\circ  A^{\ot n} -A_M\circ  f\circ(A^{\ot n-1}\ot \mathrm{id})
 \end{align*} for any $f\in C^n(R,M)$.

 \medskip

\begin{prop}The morphism $\Phi: C^\bullet(R,M)\rightarrow C^\bullet_{AO}(R,M)$ is compatible with differentials.
	\end{prop}
\begin{proof} The equalities
$$\Phi^1\circ \delta^0 =\partial^0\circ \Phi^0  $$
and $$\Phi^2\circ \delta^1=\begin{pmatrix}\partial_r^1\\\partial_l^1\end{pmatrix}\circ \Phi^1  $$
are easy to verify.

For $n\geq 2$,  we need to show
$$ \begin{pmatrix}\Phi_r^{n+1}\\\Phi_l^{n+1}\end{pmatrix}\circ \delta^n=\begin{pmatrix}\partial_r^n&0\\0&\partial_l^n\end{pmatrix}\circ \begin{pmatrix}\Phi_r^n\\\Phi_l^n\end{pmatrix},$$
that is,  $$  \Phi_r^{n+1} \circ \delta^n= \partial_r^n \circ  \Phi_r^n $$  and
 $$  \Phi_l^{n+1} \circ \delta^n= \partial_r^n \circ  \Phi_l^n. $$
 We only prove the equality  the first one,
 the second being similar.

 In fact, for $f\in \Hom(R^{\otimes n+1}, M)$ and $x_1, \cdots, x_{n+1}\in R$, we have
  $$ \begin{array}{rcl}   \Phi_r^{n+1} (\delta^n f) (x_{1, n+1})&=& \delta^n (f)(A(x)_{1, n+1})-A_M\circ \delta^n(f)(x_1\otimes A(x)_{2, n+1})\\
  &=&A(x_1)f(A(x)_{2, n+1})+\sum_{i=1}^n (-1)^i f(A(x)_{1, i-1}\otimes A(x_i)A(x_{i+1})\otimes  A(x)_{i+2, n+1})\\
  && +(-1)^{n+1}f(A(x)_{1, n})A(a_{n+1})\\
  &&-A_M(x_1f(A(x)_{2, n+1})) +A_M\circ f(x_1A(x_2)\otimes A(x)_{3, n+1})\\
  &&-\sum_{i=2}^n(-1)^i A_M(f(x_1\otimes A(x)_{2, i-1}\otimes A(x_i)A(x_{i+1})\otimes A(x)_{i+2, n+1}))\\
  &&-(-1)^{n+1} A_M\circ f(x_1\otimes A(x)_{2, n})A(x_{n+1}),

  \end{array} $$
  where we use the notations: for $i\leq j$, $x_{i, j}=x_i\otimes  \cdots  \otimes x_j$ and
  $A(x)_{i, j}=A(x_i)\otimes  \cdots  \otimes A(x_j)$; for $i>j$, they are $1\in \bfk$.
  On the other hand,
  $$ \begin{array}{rcl} \partial_r^n \circ  \Phi_l^n(x_{1, n+1})&=& A(a_1)\Phi^n_r(f)(x_{2, n+1})-A_M(x_1\Phi^n_r(f)(x_{2, n+1}))\\
  && +\sum_{i=1}^n (-1)^i \Phi^n_r(f)(x_{1, i-1}\otimes x_iA(x_{i+1})\otimes x_{i+2, n+1})\\
  && +(-1)^{n+1}\Phi^n_r(f)(x_{1, n})A(x_{n+1})\\
  &=& A(x_1) f(A(x)_{2, n+1})-A(a_1)A_M\circ f(a_2\otimes A(a)_{3, n+1}) \\
  && -A_M(x_1f(A(x)_{2, n+1}))+A_M(x_1f(x_2\otimes A(a)_{3, n+1}))\\
  && +\sum_{i=1}^n f(A(a)_{1, i-1}\otimes A(x_iA(x_{i+1}))\otimes A(x)_{i+2, n+1})\\
  && -\sum_{i=2}^n (-1)^i A_M\circ f(x_1\otimes A(x)_{2, i-1}\otimes A(x_iA(x_{i+1}))\otimes A(x)_{i+2, n+1})\\
  &&+A_M(f(x_1A(x_2)\otimes A(x)_{3, n+1})+(-1)^{n+1} f(A(x)_{1, n})A(x_{n+1})\\
  &&-(-1)^{n+1} A_M\circ f(x_1\otimes A(x)_{2, n})A(x_{n+1}).

   \end{array} $$
   We obtain $ \Phi_r^{n+1} (\delta^n f) (x_{1, n+1})=\partial_r^n \circ  \Phi_l^n(x_{1, n+1})$, because
   $$A(a_1)A_M\circ f(a_2\otimes A(a)_{3, n+1})=A_M(x_1f(x_2\otimes A(a)_{3, n+1})) $$ and
   $ A(x_iA(x_{i+1})= A(x_i)A(x_{i+1})$.

   The proof is done.
\end{proof}

\medskip

Multiplying $\Phi^n$ by $(-1)^n$, then the above commutative diagram becomes a bicomplex. Taking its totalization, we obtain a cochain complex, and denote it by $C^\bullet_{\AvA}(R,M)$.
\begin{defn} The cohomology of the cochain complex $C_{\AvA}^\bullet(R)$, denoted by $\rmH_{\AvA}^\bullet(R,M)$ is called the cohomology of the averaging algebra $(R,\cdot, A)$ with coefficients in bimodule $(M,A_M)$. When $(M,A_M)=(R,A)$,  $\rmH^\bullet_{\AvA}(R,R)$ is called the cohomology of averaging algebra $(R,\cdot,A)$ and denoted by $\rmH^\bullet_{\AvA}(R)$.
	\end{defn}

\medskip

\subsection{Relationship among the cohomlogies.}

By the commutative diagram in the last subsection, we have a canonical short exact sequences of complexes:
\[0\rightarrow C_{AO}^{\bullet -1}(R,M)\rightarrow C^\bullet_{AvA}(R,M)\rightarrow C^\bullet(R,M)\rightarrow 0. \]
Then it is trivial to obtain the following result:
\begin{theorem}
	We have the following long exact sequence of cohomology groups,
	\[\cdots\rightarrow \rmH^{n-1}_{\AvO}(R,M)\rightarrow \rmH_{\AvA}^n(R,M)\rightarrow \rm{HH}^n(R,M)\rightarrow \rmH^n_{\AvO}(R,M)\rightarrow \cdots .\]
	\end{theorem}

\section{Formal deformation of averaging algebras}
Let $(R,\bullet,A)$ be an averaging algebra. Denote by $\mu_A$ the
multiplication $\bullet$ of $R$. Consider the 1-parameterized family
$$\mu_t=\sum_{i=0}^\infty\mu_it^i,\ \mu_i\in C^2_{\Alg}(R), \
A_t=\sum_{i=0}^\infty A_it^i, A_i\in C^1_{\AvO}(R).$$

\begin{defn} A {\bf 1-parameter formal deformation} of an averaging
	algebra $(R,\mu,A)$ is a pair $(\mu_t, A_t)$ which endows the
	$\bk [[t]]$-module $(R[[t]],\mu_t,A_t)$ with the averaging
	algebra structure over $\bk[[t]]$ such $(\mu_0,A_0)=(\mu, A)$.
\end{defn}

 Power series $\mu_t$ and $ A_t$ determines a 1-parameter formal deformation of the averaging algebra $(R, \mu, A)$ if and only if for all $x,y,z\in R$, the following equations hold:
\begin{align*}
 \mu_{t}(\mu_t(x\ot y)\ot z)&=\mu_t(x\ot \mu_t(y\ot z)),\\
 \mu_t(A_t(x)\ot A_t(y))&=A_t(\mu_t(A_t(x)\ot y)),\\
  \mu_t(A_t(x)\ot A_t(y))&=A_t(\mu_t(x\ot A_t(y))).
 \end{align*}
 Expand these equations and compare the coefficients of $t^n$, we get the conditions that $\{\mu_i\}_{i\in \mathbb{N}}$ and $\{A_i\}_{i\in \mathbb{N}}$ should satisfy:
 \begin{eqnarray}
 \label{ass}\sum_{i=0}^n\mu_i\circ(\mu_{n-i}\ot \id)&=&\sum_{i=0}^n\mu_i\circ(\id\ot \mu_{n-i}),\\
 \label{avol}\sum_{i+j=0}^n\mu_{n-i-j}\circ(A_i\ot A_j)&=&\sum_{i+j=0}^nA_{n-i-j}\circ\mu_j\circ(A_i\ot \id),\\
\label{avor}  \sum_{i+j=0}^n\mu_{n-i-j}\circ(A_i\ot A_j)&=&\sum_{i+j=0}^nA_{n-i-j}\circ\mu_j\circ(\id\ot A_i).
 \end{eqnarray}

 \begin{prop}
 	Let $(R[[t]], \mu_t, A_t)$ be a 1-parameter formal deformation of an averaging algebra $(R,\mu,A)$. Then $(\mu_1, A_1)$ is a 2-cocycle in  the cochain complex $C^\bullet_{\AvA}(R)$.
 	\end{prop}
 \begin{proof}
 	Compute the equations (\ref{ass}) (\ref{avol}) (\ref{avor}) for $n=1$, we have :
 	\begin{eqnarray*}
 		\mu_1\circ(\mu\ot \id)+\mu\circ(\mu_1\ot \id)&=&\mu_1\circ(\id\ot \mu)+\mu\circ(\id\ot \mu_1),\\
 		\mu_1\circ(A\ot A)+\mu\circ(A\ot A_1)+\mu\circ(A_1\ot A)&=&A_1\circ\mu\circ(A\ot\id)+A\circ\mu\circ(A_1\ot\id)+A\circ\mu_1\circ(A\ot \id),	\\
 		\mu_1\circ(A\ot A)+\mu\circ(A\ot A_1)+\mu\circ(A_1\ot A)&=&A_1\circ\mu\circ(\id\ot A)+A\circ \mu\circ(\id\ot A_1)+A\circ \mu_1\circ(\id\ot A).
 	\end{eqnarray*}

 They are equivalent to :
 \begin{eqnarray*}
 	\delta(\mu_1)&=&0,\\
 	\partial_r(A_1)+\Phi_r^2(\mu_1)&=&0,\\
 	\partial_l(A_1)+\Phi_l^2(\mu_1)&=&0.
 	\end{eqnarray*}

That is, $(\mu_1, A_1)$ is a 2-cocycle in $C^\bullet_{\AvA}(R)$.
 	\end{proof}

 If $\mu_t=\mu_A$ in the above 1-parameter formal deformation of the averaging algebra $(R,\mu,A)$, we get a 1-parameter formal deformation of the averaging operator $A$. Consequently, we have:
 \begin{coro}Let $A_t$ be a 1-parameter formal deformation of the averaging operator $A$. Then $A_1$ is a 1-cocycle in the cochain complex $C_{\AvO}^\bullet(R)$.
 	\end{coro}
 \begin{defn}The 2-cocycle $(\mu_1, A_1)$ is called the infinitesimal of the 1-parameter formal deformation $(R[[t]], \mu_t, A_t)$ of averaging $(R,\mu,A)$.
 \end{defn}
\begin{defn}Let $(R[[t]],\mu_t, A_t)$ and $(R[[t]],\mu^\prime_t, A^\prime_t)$ be two 1-parameter for deformations of averaging algebra $(R,\mu, A)$. A formal isomorphism from $(R[[t]],\mu^\prime_t, A^\prime_t )$ to $(R[[t]],\mu_t, A_t)$ is a power series $\phi_t=\sum\limits_{i\geqslant 0}\phi_it^i: R[[t]]\rightarrow R[[t]]$, where $\phi_i: R\rightarrow R, i\in \mathbb{N}$ are linear maps with $\phi_0=\id$, such that
\begin{eqnarray}
\phi_t\circ\mu'_t&=&\mu_t\circ(\phi_t\ot\phi_t),\\
\phi_t\circ A_t'&=&A_t\circ \phi_t.
\end{eqnarray}
Two 1-parameter formal deformations $(R[[t]], \mu_t, A_t)$ and $(R[[t]], \mu'_t, A'_t)$ are said to be equivalent if there exists a formal isomorphism $\phi_t:(R[[t]], \mu'_t, A'_t)\rightarrow (R[[t]], \mu_t, A_t)$.
	\end{defn}

\begin{theorem}\label{thm: equivalent infini in same cohomology}
	The infinitesimals of two equivalent 1-parameter formal deformations of $(R,\mu,A)$ are in the same cohomology class in $\rmH^2_{\AvA}(R).$ Conversely, if the infinitesimals of two   1-parameter formal deformations of $(R,\mu,A)$ fall  into the same cohomology class, they are equivalent.
\end{theorem}
\begin{proof}
	Let $\phi_t: (R[[t]], \mu'_t, A_t')\rightarrow (R[[t]], \mu_t, A_t)$ be a formal isomorphism. For all $x,y\in R$, we have
	\begin{eqnarray}\phi_t\circ\mu'_t(x\ot y)&=&\mu_t\circ(\phi_t\ot \phi_t)(x\ot y),\\
	\phi_t\circ A_t'(x)&=&A_t\circ \phi_t(x).
	\end{eqnarray}
	
	Expanding the identities comparing the coefficients of $t$, we get:
	\begin{eqnarray}\mu_1'(x\ot y)&=&\mu_1(x\ot y)+x\phi_1(y)-\phi_1(xy)+\phi_1(x)y,\\
	A_1'(x)&=&A_1(x)+A(\phi_1(x))-\phi_1(A(x)).
	\end{eqnarray}
	So we have, \[(\mu_1',A_1')-(\mu_1,A_1)=d(\phi_1) \in C^\bullet_{AvA}(R).\]
	That is, $[(\mu_1',A_1')]=[(\mu_1,A_1)] \in \rmH^2_{\AvA}(R)$.

Conversely, given two formal deformations $ (R[[t]], \mu'_t, A_t')$ and $ (R[[t]], \mu_t, A_t)$   of an averaging algebra $(R,\mu,A)$, suppose that $(\mu_1, A_1)$ and $(\mu_1', A_1')$ are in the same cohomology class of  $\rmH^2_{\AvA}(R).$
	\end{proof}
\begin{defn}
	A $1$-parameter formal deformation $(R[[t]], \mu_t, A_t)$ of $(R,\mu,A)$ is said to be trivial if it is equivalent to be deformation $(R[[t]],\mu,A)$, that is, there exists $\phi_t=\sum\limits_{i=0}^\infty\phi_it^i: R[[t]]\rightarrow R[[t]]$, where $\phi: A\rightarrow A$ are linear maps with $\phi=\id$, such that
	\begin{eqnarray}
	\phi_t\circ \mu_t&=&\mu_A\circ (\phi_t\ot \phi_t),\\
	\phi_t\circ A_t&=&A\circ \phi_t.
	\end{eqnarray}
\end{defn}

\begin{defn}
	An averaging algebra $(R,\mu,A)$ is said to be rigid if every $1$-parameter formal deformation is trivial.
\end{defn}
 
\begin{prop}
	Let $(R,\mu,A)$ be an averaging algebra. If $\rmH^2_{\AvA}(R)=0$, $(R,\mu,A)$ is rigid.
\end{prop}

\begin{proof}

Let $(R[[t]], \mu_t, A_t)$ be a $1$-parameter formal deformation of $(R, \mu, A)$. By Proposition~,
$(\mu_1, A_1)$ is a $2$-cocycle. By $\rmH^2_{\AvA}(R)=0$, there exists a $1$-cochain $(\phi_1', x) \in  C^1_{\Alg}(R, R)\oplus \Hom(k, M)$ such that
$(\mu_1, A_1) = -d_{\AvA}(\phi_1', x), $
that is, $\mu_1=-\delta^1(\phi_1')$ and $A_1=-\partial^(x)-\Phi^1(\phi_1')$. Since $\Phi^0=Id$, let $\phi_1=\phi_1'+\delta^0(x)$. Then 
 $\mu_1=-\delta^1(\phi_1)$ and $A_1=-\Phi^1(\phi_1')$.

Setting $\phi_t = Id_R +\phi_1t$, we have a deformation $(R[[t]], \overline{\mu}_t, \overline{A}_t)$, where
$$\overline{\mu}_t=\phi_t^{-1}\circ \mu_t\circ (\phi_t\times \phi_t)$$
and $$\overline{A}_t=\phi_t^{-1}\circ A_t\circ \phi_t.$$
 It is not difficult to see that 
 $$\begin{array}{rcl} \overline{\mu}_t&=& \mu+\overline{\mu}_2t^2+\cdots,\\
 A_t&=& A+\overline{A}_2t^2+\cdots.\end{array}$$
Then by repeating the argument, we can show that $ (R[[t]], \mu_t , A_t) $ is equivalent to the trivial extension $(R[[t]], \mu, A).$
Thus, $(R,\mu,A)$ is rigid.

\end{proof}
 \section{Abelian extensions of averaging algebras}

 In this section, we study abelian extensions of averaging algebras and show that they are classified by the second cohomology, as one would expect of a good cohomology theory.

 \begin{defn}
 	An {\bf abelian extension} of averaging algebras is a short exact sequence of homomorphisms of averaging algebras
 	\[\begin{CD}
 		0@>>> {M} @>i >> \hat{R} @>p >> R @>>>0\\
 		@. @V {A_M} VV @V {\hat{A}} VV @V A VV @.\\
 		0@>>> {M} @>i >> \hat{R} @>p >> R @>>>0
 	\end{CD}\]
 	such that $uv=0$ for all $u,v\in M.$
 \end{defn}

 We will call $(\hat{R},\hat{A})$ an abelian extension of $(R,A)$ by $(M,A_M)$.

 \begin{defn}
 	Let $(\hat{R}_1,\hat{A}_1)$ and $(\hat{R}_2,\hat{A}_2)$ be two abelian extensions of $(R,A)$ by $(M,A_M)$. They are said to be {\bf isomorphic} if there exists an isomorphism of averaging algebras $\zeta:(\hat{R}_1,\hat{A}_1)\rar (\hat{R}_2,\hat{A}_2)$ such that the following commutative diagram holds:
 	\[\begin{CD}
 		0@>>> {(M,A_M)} @>i >> (\hat{R}_1,{\hat{A}_1}) @>p >> (R,A) @>>>0\\
 		@. @| @V \zeta VV @| @.\\
 		0@>>> {(M,A_M)} @>i >> (\hat{R}_2,{\hat{A}_2}) @>p >> (R,A) @>>>0.
 	\end{CD}\]
 \end{defn}

 A {\bf section} of an abelian extension $(\hat{R},{\hat{A}})$ of $(R,A)$ by $(M,A_M)$ is a linear map $s:R\rar \hat{R}$ such that $p\circ s=\Id_R$.

 Now for an abelian extension $(\hat{R},{\hat{A}})$ of $(R,A)$ by $(M,A_M)$ with a section $s:R\rar \hat{R}$, we define linear maps $\rho_l: R\to \mathrm{End}_\bk(M),~ r\mapsto (m\mapsto rm)$ and $\rho_r: R\to \mathrm{End}_\bk(M),~ r\mapsto (m\mapsto mr)$ respectively by
 $$
 rm:=s(r)m,\quad mr:=ms(r), \quad \forall r\in R, m\in M.
 $$
 \begin{prop}
 	With the above notations, $(M,\rho_l,\rho_r,A_M)$ is a bimodule over the averaging algebra $(R,A)$.
 \end{prop}
 \begin{proof}
 	For any $x,y\in R,\,v\in M$, since $s(xy)-s(x)s(y)\in M$ implies $s(xy)m=s(x)s(y)m$, we have
 	\[\rho_l(xy)(m)=s(xy)m=s(x)s(y)m=\rho_l(x)\circ\rho_l(y)(m).\]
 	Hence, $\rho_l$ is an algebra homomorphism. Similarly, $\rho_r$ is an algebra anti-homomorphism. Moreover, ${\hat{A}}(s(r))-s(A(r))\in M$ means that  ${\hat{A}}(s(r))m=s(A(r))m$. Thus we have
 	\begin{align*}
 		A(r)A_M(m)&=s(A(r))A_M(m)\\
 		&=\hat{A}(s(r))A_M(m)\\
 		&=\hat{A}(s(r)A_M(m))=\hat{A}(\hat{A}(sr)m)\\
 		&=A_M(rA(m))=A_M(A(r)m).
 	\end{align*}
 It is similar to see $A_M(m)A(r)=A_M(A_M(m)r)=A_M(mA(r))$.
 	Hence, $(M,\rho_l,\rho_r,A_M)$ is a bimodule over $(R,A)$.
 \end{proof}

 We  further  define linear maps $\psi:R\otimes R\rar M$ and $\chi:R\rar M$ respectively by
 \begin{align*}
 	\psi(x\ot y)&=s(x)s(y)-s(xy),\quad\forall x,y\in R,\\
 	\chi(x)&={\hat{A}}(s(x))-s(A(x)),\quad\forall x\in R.
 \end{align*}
 We transfer the averaging algebra structure on $\hat{R}$ to $R\oplus M$ by endowing $R\oplus M$ with a multiplication $\cdot_\psi$ and an averaging operator $A\chi$ defined by
 \begin{align}
 	\label{eq:mul}(x,m)\cdot_\psi(y,n)&=(xy,xn+my+\psi(x,y)),\,\forall x,y\in R,\,m,n\in M,\\
 	\label{eq:dif}A_\chi(x,m)&=(A(x),\chi(x)+A_M(m)),\,\forall x\in R,\,m\in M.
 \end{align}

 \begin{prop}\label{prop:2-cocycle}
 	The triple $(R\oplus M,\cdot_\psi,A_\chi)$ is a averaging algebra   if and only if
 	$(\psi,\chi)$ is a 2-cocycle  of the averaging algebra $(R,A)$ with the coefficient  in $(M,A_M)$.
 \end{prop}
 \begin{proof}
 	If $(A\oplus M,\cdot_\psi,A_\chi)$ is a averaging algebra, then the associativity of $\cdot_\psi$ implies
 	\begin{equation}
 		\label{eq:mc}x\psi(y\ot z)-\psi(xy\ot z)+\psi(x\ot yz)-\psi(x\ot y)z=0,
 	\end{equation}
 which means $\delta(\phi)=0$ in $C^\bullet(A,M)$.
 Since $A_\chi$ is an averaging operator,
 for any $x,y\in R, m,n\in M$, we have
 $$A_\chi((x,m))\cdot_\psi A_\chi((y,n))=A_\chi(A_\chi(x,m)\cdot_\psi(y,n))=A_\chi((x,m)\cdot_\psi A_\chi(y,n))$$
 	Then $\chi,\psi$ satisfy the following equations:
 	\begin{align*}
 		[\psi(A(x)\ot A(y))-A_M(\psi(A(x)\ot y))]+[A(x)\chi(y)-A_M(\chi(x)y)-\chi(A(x)y)+\chi(x)A(y)]&=0\\
 		[\psi(A(x)\ot A(y))-A_M(\psi(x\ot A(y)))]+[A(x)\chi(y)-A_M(x\chi(y))-\chi(xA(y))+\chi(x)A(y)]&=0
 	\end{align*}
 	
 	That is,
 	\[\partial_r^1(\chi)+\Phi_r^2(\psi)=0, \partial_l^1(\chi)+\Phi_l^2(\psi)=0.\]
 Hence, $(\psi,\chi)$ is a  2-cocycle.
 	
 	Conversely, if $(\psi,\chi)$ is a 2-cocycle, one can easily check that $(R\oplus M,\cdot_\psi,A_\chi)$ is an averaging algebra.
 \end{proof}

 Now we are ready to classify abelian extensions of an averaging algebra.

 \begin{theorem}
 	Let $M$ be a vector space and  $A_M\in\End_\bk(M)$.
 	Then abelian extensions of an averaging algebra $(R,A)$ by $(M,A_M)$ are classified by the second cohomology group ${{H}}_{\AvA}^2(R,M)$ of $(R,A)$ with coefficients in the bimodule $(M,A_V)$.
 \end{theorem}
 \begin{proof}
 	Let $(\hat{R},{\hat{A}})$ be an abelian extension of $(R,A)$ by $(M,A_M)$. We choose a section $s:R\rar \hat{R}$ to obtain a 2-cocycle $(\psi,\chi)$ by Proposition~\ref{prop:2-cocycle}. We first show that the cohomological class of $(\psi,\chi)$ does not depend on the choice of sections. Indeed, let $s_1$ and $s_2$ be two distinct sections providing 2-cocycles $(\psi_1,\chi_1)$ and $(\psi_2,\chi_2)$ respectively. We define $\phi:R\rar M$ by $\gamma(r)=s_1(r)-s_2(r)$. Then
 	\begin{align*}
 		\psi_1(x,y)&=s_1(x)s_1(y)-s_1(xy)\\
 		&=(s_2(x)+\gamma(x))(s_2(y)+\gamma(y))-(s_2(xy)+\gamma(xy))\\
 		&=(s_2(x)s_2(y)-s_2(xy))+s_2(x)\gamma(y)+\gamma(x)s_2(y)-\gamma(xy)\\
 		&=(s_2(x)s_2(y)-s_2(xy))+x\gamma(y)+\gamma(x)y-\gamma(xy)\\
 		&=\psi_2(x,y)+\delta(\gamma)(x,y)
 	\end{align*}
 	and
 	\begin{align*}
 		\chi_1(x)&={\hat{A}}(s_1(x))-s_1(A(x))\\
 		&={\hat{A}}(s_2(x)+\gamma(x))-(s_2(A(x))+\gamma(A(x)))\\
 		&=({\hat{A}}(s_2(x))-s_2(A(x)))+{\hat{A}}(\gamma(x))-\gamma(A(x))\\
 		&=\chi_2(x)+d_V(\gamma(x))-\gamma(d_A(x))\\
 		&=\chi_2(x)-\Phi^1(\gamma)(x).
 	\end{align*}
 	That is, $(\psi_1,\chi_1)=(\psi_2,\chi_2)+d(\gamma)$. Thus $(\psi_1,\chi_1)$ and $(\psi_2,\chi_2)$ are in the same cohomological class  {in $H_{\AvA}^2(R,M)$}.
 	
 	Next we prove that isomorphic abelian extensions give rise to the same element in  {$H_{\AvA}^2(R,M)$.} Assume that $(\hat{R}_1,{\hat{A}_1})$ and $(\hat{R}_2,{\hat{A}_2})$ are two isomorphic abelian extensions of $(R,A)$ by $(M,A_M)$ with the associated homomorphism $\zeta:(\hat{R}_1,{\hat{A}_1})\rar (\hat{R}_2,{\hat{A}_2})$. Let $s_1$ be a section of $(\hat{R}_1,{\hat{A}_1})$. As $p_2\circ\zeta=p_1$, we have
 	\[p_2\circ(\zeta\circ s_1)=p_1\circ s_1=\Id_{R}.\]
 	Therefore, $\zeta\circ s_1$ is a section of $(\hat{R}_2,{\hat{A}_2})$. Denote $s_2:=\zeta\circ s_1$. Since $\zeta$ is a homomorphism of differential algebras such that $\zeta|_M=\Id_M$, we have
 	\begin{align*}
 		\psi_2(x\ot y)&=s_2(x)s_2(y)-s_2(xy)=\zeta(s_1(x))\zeta(s_1(y))-\zeta(s_1(xy))\\
 		&=\zeta(s_1(x)s_1(y)-s_1(xy))=\zeta(\psi_1(x,y))\\
 		&=\psi_1(x,y)
 	\end{align*}
 	and
 	\begin{align*}
 		\chi_2(x)&={\hat{A}_2}(s_2(x))-s_2(A(x))={\hat{A}_2}(\zeta(s_1(x)))-\zeta(s_1(A(x)))\\
 		&=\zeta({\hat{A}_1}(s_1(x))-s_1(A(x)))=\zeta(\chi_1(x))\\
 		&=\chi_1(x).
 	\end{align*}
 	Consequently, all isomorphic abelian extensions give rise to the same element in {${H}_{\AvA}^2(R,M)$}.
 	
 	Conversely, given two 2-cocycles $(\psi_1,\chi_1)$ and $(\psi_2,\chi_2)$, we can construct two abelian extensions $R\oplus M,\cdot_{\psi_1},A_{\chi_1})$ and  $(R\oplus M,\cdot_{\psi_2},A_{\chi_2})$ via equalities~\eqref{eq:mul} and \eqref{eq:dif}. If they represent the same cohomological class {in $H_{\AvA}^2(R,M)$}, then there exists a linear map $\gamma:R\to M$ such that $$(\psi_1,\chi_1)=(\psi_2,\chi_2)+(\delta(\gamma),\Phi^1(\gamma)).$$ Define $\zeta:R\oplus M\rar R\oplus M$ by
 	\[\zeta(r,m):=(r,\gamma(r)+m).\]
 	Then $\zeta$ is an isomorphism of these two abelian extensions.
 \end{proof}

\section{$L_\infty$-structure on the cochain complex}

Let's recall the definition of $L_\infty$-algebras.

For graded indeterminates $t_1,\dots,t_n$ and $\sigma\in S_n$, the Koszul sign $\epsilon(\sigma, t_1,\dots, t_n)$ is defined by
$$t_1\cdot t_2 \dots\cdot t_n=\epsilon(\sigma, t_1,\dots,
t_n)t_{\sigma(1)}\cdot t_{\sigma(2)}\dots\cdot t_{\sigma(n)},$$
where $``\cdot"$ is the multiplication in the free graded commutative algebra $k\langle t_1,\dots,t_n\rangle/(t_it_j-(-1)^{|t_i||t_j|}t_jt_i)$ generated by $t_1,\dots,t_n$. Define  also $\chi(\sigma,t_1,\dots,t_n)=sgn(\sigma)\epsilon(\sigma,t_1,\dots,t_n)$.

\begin{defn}{\textrm{(\cite{JS90,LS93,LM02})}}\label{def-L-infty}  Let $L=\bigoplus\limits_{i\in \mathbb{Z}}L_i$ be a graded
	space over $\bk$.  Assume that $L$ is endowed with a family of linear
	operators $\{l_n:L^{\ot n}\rightarrow L\}_{n\geqslant 1}$ with $|l_n|=n-2$ satisfying the following conditions:
	\begin{itemize}
		\item[(i)] $l_n(x_{\sigma(1)}\ot \dots\ot x_{\sigma(n)})=\chi(\sigma,x_1,\dots,x_n)l_n(x_1\ot \dots\ot x_n)$ , $\forall \sigma\in S_n$, $x_1,\dots,x_n\in L$,
		\item[(ii)]$\sum\limits_{i=1}^n\sum\limits_{\sigma\in { S}(i,n-i)}\chi(\sigma,x_1,\dots,x_n)(-1)^{i(n-i)}l_{n-i+1}(l_i(x_{\sigma(1)}\ot \dots\ot x_{\sigma(i)})\ot x_{\sigma(i+1)}\ot \dots\ot x_{\sigma(n)})=0,$
		where $S(i,n-i)$ is the set of all $(i,n-i)$-shuffles, i.e., $S(i,n-i)=\{\sigma\in S_n,|\sigma(1)<\sigma(2)<\dots<\sigma(i),\ \sigma(i+1)<\sigma(i+2)<\dots<\sigma(n)\}$,
	\end{itemize}
	Then $(L,\{l_n\}_{n\geqslant 1})$ is called a $L_\infty$-algebra.
\end{defn}

Let $sL$ be the suspension of $L$, i.e., $(sL)_{n}=L_{n-1}$.
Let  $S^\bullet(sL)$ be the cofree cocommutative coalgebra  generated by $sL$
. Then $\{l_n\}_{n\geqslant 1}$ will
determine a family of operators $\{d_{n}\}_{n\geqslant 1}$, where $d_n:(sL)^{\ot n}\rightarrow sL$ is
defined as $d_n=(-1)^{\frac{n(n-1)}{2}}s\circ l_n\circ ({s^{-1}}^{\ot n})$. Then $\{d_n\}_{n\geqslant 1}$ satisfies the equation
\begin{eqnarray}\label{coderivation}
	\sum_{\sigma\in S(i,n-i)}\varepsilon(\sigma,sx_1,\dots,sx_n)d_{n-i+1}(d_i(sx_{\sigma(1)},\dots,sx_{\sigma(i)}),sx_{\sigma(i+1)},\dots,sx_{\sigma(n)})=0.
\end{eqnarray}

Denote $\sum\limits_{i=1}^\infty d_n:S^\bullet(sL)\rightarrow sL$ by $d'$. Then $d'$ can induce a coderivation $d$ on $S^\bullet(sL)$. Equation (\ref{coderivation}) ensures that the coderivation $d$ is a differential, i.e., $d^2=0$. This can be considered as an equivalent definition of $L_\infty$-algebra.

\begin{defn}Let $(L,\{l_n\}_{n\geqslant 1})$ be a $L_\infty$-algebra and $\alpha\in L_{-1}$. We call $\alpha$ a Maurer-Cartan element if it satisfies the following equation:
	$$\sum_{n=1}^{\infty}\frac{1}{n!}(-1)^{\frac{n(n-1)}{2}}l_n(\alpha^{\ot n})=0.$$
\end{defn}

\begin{prop}Given a Maurer-Cartan element $\alpha$ in
	$L_\infty$-algebra $L$, we can define a new
	$L_\infty$-structure $\{l^\alpha_n\}_{n\geqslant 1}$ on $L$, where $l^\alpha_n:L^{\ot n}\rightarrow L$ is defined as:
	$$l^\alpha_n(x_1\ot \dots\ot x_n)=\sum_{i=0}^\infty\frac{1}{i!}(-1)^{in+\frac{i(i+1)}{2}+\sum\limits_{k=1}^{n-1}\sum\limits_{j=1}^k|x_j|}l_{n+i}(\alpha^{\ot i}\ot x_1\ot \dots\ot x_n).$$
\end{prop}

Let $\alpha$ be a Maurer Cartan element in $L_\infty$ algebra $L$. Use the equivalent definition of $L_\infty$-algebra, we can see $s\alpha\in sL$ satisfies the following equation on $S^\bullet(sL)$:
\[\sum_{n=1}^\infty \frac{1}{n!}d_{n}\big((s\alpha)^{\ot n}\big)=0.\] Define \[d_n^\alpha(sx_1\ot \dots sx_n)=\sum_{i=0}^\infty\frac{1}{i!}d_{n+1}\big((s\alpha)^{\ot i}\ot sx_1\ot \dots \ot sx_n\big).\]
Then the family $\{d_n^\alpha\}_{n\geqslant 1}$ can also induce a differential on $S^\bullet(sL)$.

\medskip

Let $W$ be a graded space and $T^c(W)$ be the cofree coalgebra generated by $W$. For $f\in \Hom(W^{\ot n},W),g_i\in \Hom(W^{\ot l_i},W), 1\geqslant i\geqslant m$,  $f\bar\circ(g_1,\dots,g_m)\in \Hom(W^{\ot l_1+\dots+l_m+n-m},W)$   is defined as follows:

\begin{align*}&\big(f\bar\circ(g_1,\dots,g_m)\big)(w_1\ot \dots \ot w_{l_1+\dots+l_m+n-m})=\\
	&\sum_{ 0\leqslant i_1\leqslant i_2\leqslant\dots\leqslant i_m} (-1)^{\eta}f\Big(w_1\ot \dots w_{i_1}\ot g_{1}(w_{i_1+1}\ot \dots )\ot\dots g_{k}(w_{i_{k}+1}\ot \dots ) \ot\dots \ot  g_m(w_{i_m+1}\ot \dots )\ot \dots \Big)
\end{align*}
where $\eta=\sum\limits_{k=1}^m|g_k|(\sum\limits_{j=1}^{i_k}|w_j|)$.

For any $f,g\in \Hom(T^c(W),W)$,  their Gerstenhaber bracket $[f,g]_G\in \Hom(T^c(W),W)$ is defined as :
\[[f,g]_G=f\bar\circ g-(-1)^{|f||g|}g\bar\circ f.\]

\begin{lemma}{\cite{Ger63}}
	Let $W$ be a graded space. Then $(\Hom(T^c(W),W),[-,-]_G)$ forms a graded Lie-algebra.
\end{lemma}

\medskip

Now, given a graded space $V=\bigoplus\limits_{i\in \mathbb{Z}}V_i$,
we will define a $L_\infty$-algebra $\mathcal{C}_{AvA}(V)$. We
will see that when $V$ is concentrated in degree 0, an averaging
algebra structure on $V$ is equivalent to a Maurer-Cartan element in
this $L_\infty$-algebra.

\medskip

Firstly, the underlying graded space of $\mathcal{C}_{AvA}(V)$ is $$\calc_A(V)\oplus \Hom(k,V)\oplus \Hom(sV,V)\oplus \calc_{AvO}(V)_l^{\geqslant 2}\oplus \calc_{AvO}(V)_r^{\geqslant 2}.$$

where \begin{align*}
\calc_A(V)&=\Hom(T^c(sV),sV),\\
\calc_{AvO}(V)_l^{\geqslant 2}&=\Hom(\bigoplus\limits_{n=2}^\infty(sV)^{\ot n},V),\\
\calc_{AvO}(V)_r^{\geqslant 2}&=\Hom(\bigoplus\limits_{n=2}^\infty(sV)^{\ot n},V).
\end{align*}

And we define
\begin{align*}{\calc_ {AvO}}(V)_r&=\Hom(k,V)\oplus \Hom(sV,V)\oplus{\calc_{AvO}}(V)_r^{\geqslant 2}\\
	     {\calc_{AvO}}(V)_l&=\Hom(k,V)\oplus \Hom(sV,V)\oplus {\calc_{AvO}}(V)_l^{\geqslant 2}
	\end{align*}

\medskip

Now, let's give a $L_\infty$-algebra structure on $\calc_{AvA}(V)$, i.e., we need to give a family of operators $\{l_n\}_{n\geqslant 1}$ to satisfy the conditions in   Definition~\ref{def-L-infty}. Here, we identify $\Hom(T^c(sV),sV)$ with $s\Hom(T^c(sV),V)$.

The family $\{l_i\}_{i \geqslant 1}$ are defined by the following processes:
\begin{enumerate}
	\item[(I)] For $sh\in \Hom(k,sV)\subset \calc_A(V)$, $l_1(sh)=h$. And   $l_1$ vanishes elsewhere.
	
	\medskip
	
	\item[(II)] For $sg,sh\in \calc_A(V)$, $l_2(sg\ot sh)=[sg,sh]_{G}$.
	
	\medskip
	
	\item[(III)]For  homogeneous elements $sh\in \Hom((sV)^{\ot n},sV)$ in $\calc_{A}(V)$ and $g_1,\dots,g_n\in \calc_{AvO}(V)_r $, define $l_{n+1}^r(sh\ot g_1\ot \dots \ot g_n)\in \calc_{AvO}(V)_r$ in the following ways:
	
	\medskip
	
	\begin{itemize}
		\item[(i)]If $g_1,\dots, g_n$ are all contained in $\Hom(sV,V)\bigoplus\calc_{AvO}(V)^{\geqslant 2}_r$, then \begin{align*}&l_{n+1}^r(sh\ot g_1\ot\dots\ot g_n)=\\
			&\sum_{\sigma_\in S_n} (-1)^{\varepsilon}\Big(h\circ (sg_{\sigma(1)}\ot \dots \ot sg_{\sigma(n)})-(-1)^{(|g_{\sigma(1)}|+1)(|h|+1)}g_{\sigma(1)}\bar{\circ}(sh\circ(\id_{sV}\ot sg_{\sigma(2)}\ot \dots \ot sg_{\sigma(n)}))\Big)
		\end{align*}
		
		\smallskip
		
		\item[(ii)] { If there exists some $g_i$ coming from $\Hom(k,V)$, then
			
			\begin{align*}
				&l_{n+1}^r(sh\ot g_1\ot\dots\ot g_n)=
				\sum_{\sigma_\in S_n} (-1)^{\varepsilon}\Big(h\circ (sg_{\sigma(1)}\ot \dots \ot sg_{\sigma(n)})\\
				&-(-1)^{\big(|g_{\sigma(p)}|+1\big)\big(|h|+1+\sum\limits_{k=1}^{p-1}(|g_{\sigma(k)}|+1)\big)}g_{\sigma(p)}\bar{\circ}(sh\circ(g_{\sigma(1)}\ot sg_{\sigma(2)}\ot \dots\ot g_{\sigma(p-1)} \ot\id_{sV}\ot sg_{\sigma(p+1)}\ot \dots \ot sg_{\sigma(n)}))\Big)
			\end{align*}
		}
		where  $(-1)^\varepsilon=\chi(\sigma;g_1,\dots,g_n)\cdot(-1)^{n(|h|+1)+\sum\limits_{k=1}^{n-1}\sum\limits_{j=1}^k|g_{\sigma(j)}|}$, and $p$ is the integer such that $g_{\sigma(1)},\dots,g_{\sigma(p-1)}\in \Hom(k,V)$ and $g_{\sigma(p)}\in \Hom(sV,V)\oplus\calc_{AvO}(V)_r^{\geqslant 2}$.
		
	\end{itemize}

	\medskip
	
	\item[(IV)] Let $sh\in\Hom((sV)^{\ot n},sV)$ in $\calc_A(V)$,   $g_1,\dots,g_n \in  \calc_{AvO}(V)_l $ be homogeneous elements.  Define $l_{n+1}^l(sh\ot g_1\ot \dots \ot g_n)\in \calc_{AvO}(V)_l $  in the following ways:
	
	\smallskip
	
	\begin{itemize}
		
		\item[(i)] If $g_1, \dots,g_n$ all belong to $\Hom(sV,V)\bigoplus \calc_{AvO}(V)^{\geqslant 2}_l$, then define
		\begin{align*}
			&l_{n+1}^l(sh\ot g_1\ot \dots \ot g_n)=\\
			&\sum_{\sigma\in S_n}(-1)^{\varepsilon}\Big(h\circ(sg_{\sigma(1)}\ot \dots \ot sg_{\sigma(n)})-(-1)^{(g_{\sigma(n)}+1)(|h|+1+\sum\limits_{k=1}^{n-1}(|g_{\sigma(k)}|+1))}g_{\sigma(n)}\bar{\circ}(sh\circ(sg_{\sigma(1)}\ot \dots \ot sg_{\sigma(n-1)}\ot \id_{sV}))\Big),
		\end{align*}
		
		\smallskip
		
		\item[(ii)] {If there exists some $g_i\in \Hom(k,V)$, then define:
			\begin{align*}
				l_{n+1}^l&(sh\ot g_1\ot \dots \ot g_n)=
				\sum_{\sigma\in S_n}(-1)^{\varepsilon}\Big(h\circ(sg_{\sigma(1)}\ot \dots \ot sg_{\sigma(n)})\\
				&-(-1)^{\big(|g_{\sigma(q)}|+1\big)\big(|h|+1+\sum\limits_{k=1}^{q-1}(|g_{\sigma(k)}|+1)\big)}g_{\sigma(q)}\bar{\circ}(sh\circ(sg_{\sigma(1)}\ot \dots \ot sg_{\sigma(q-1)}\ot \id_{sV}\ot\ot g_{\sigma(q+1)}\ot \dots \ot sg_{\sigma(n)}))\Big).
		\end{align*}}
		
		Where $q$ is the integer such that $g_{\sigma(q+1)},\dots,g_{\sigma(n)}\in \Hom(k,V)$ and $g_{\sigma(q)}\in \Hom(sV,V)\oplus\calc_{AvO}(V)^{\geqslant 2}_l$.
	\end{itemize}
	
	\medskip
	
	\item[(V)]
	\begin{itemize}
		\item [(i)] If $g_1,\dots,g_n\in \Hom(k,V)\oplus\Hom(sV,V)=\calc_{AvO}(V)_l\cap\calc_{AvO}(V)_r$  with at most one $g_i\in \Hom(sV,V)$, then the definitions of $l_{n+1}^r$, $l_{n+1}^l$ coincide. Then we  define $$l_{n+1}(sh,g_1,\dots,g_n):=l_{n+1}^r(sh,g_1,\dots,g_n)=l_{n+1}^l(sh,g_1,\dots,g_n).$$
		
		\smallskip
		
		\item[(ii)] If $g_1,\dots,g_n\in \Hom(k,V)\oplus\Hom(sV,V)=\calc_{AvO}(V)_l\cap\calc_{AvO}(V)_r$  with at least two $g_i\in \Hom(sV,V)$, then $l_{n+1}^l(sh,g_1,\dots,g_n)\in \calc_{AvO}(V)_l^{\geqslant2}$ and $l_{n+1}^r(sh,g_1,\dots,g_n)\in \calc_{AvO}(V)_r^{\geqslant2}$. Then we  define $$l_{n+1}(sh,g_1,\dots,g_n)=(l_{n+1}^r(sh,g_1,\dots,g_n),l_{n+1}^l(sh,g_1,\dots,g_n))\in {\calc_{AvO}}(V)^{\geqslant 2}_r\oplus {\calc_{AvO}}(V)^{\geqslant 2}_l.$$
		
		\smallskip
		
		\item[(iii)]For $g_1,\dots,g_n\in \calc_{AvO}(V)_r$ with some $g_i\in \calc_{AvO}(V)_r^{\geqslant 2}$ , define
		$$l_{n+1}(sh,g_1,\dots,g_n):=l_{n+1}^r(sh,g_1,\dots,g_n).$$
		
		For $g_1,\dots,g_n\in \calc_{AvO}(V)_l$ with some $g_i\in\calc_{AvO}(V)_l^{\geqslant 2} $, define
		$$l_{n+1}(sh,g_1,\dots,g_n):=l_{n+1}^r(sh,g_1,\dots,g_n).$$
	\end{itemize}

	\medskip
	
	\item[(VI)] At last we define 	
	\[l_{n+1}(g_1\ot \dots \ot g_i\ot sh\ot g_{i+1}\ot \dots \ot g_n)=(-1)^{(|h|+1)(\sum\limits_{k=1}^i|g_k|)+i}l_{n+1}(sh\ot g_1\ot \dots \ot g_n).\] And $l_n$ vanishes elsewhere.
\end{enumerate}

\medskip

\begin{theorem}{\label{averaging-L-infty}}
	Let $V$ be a graded space. Then ${\calc_{AvA}}(V)$ endowed with operations $\{l_n\}_{n\geqslant 1}$ defined above forms a $L_\infty$-algebra.
\end{theorem}

Now, let's realize averaging algebra structures as Maurer-Cartan elements in this
$L_\infty$-algebra.

\begin{theorem}
	Let $R$ be a vector space. Then an averaging structure on $R$
	is equivalent to a Maurer-Cartan elements in ${\calc_{AvA}}(R)$.
\end{theorem}
\begin{proof}
	Consider $R$ as a graded vector space concentrated in degree 0.
	Then the degree -1 part of $\calc_{AvA}(R)$ is $${\calc_{AvA}}(R)_{-1}=\Hom((sV)^{\ot 2},
	sV)\bigoplus \Hom(sV,V).$$ Assume that $\alpha=(m, \tau)$ be a Maurer-Cartan element in $\calc_{AvA}(R)$, that is, $\alpha$ satisfy the equation:
	\begin{eqnarray}
		\label{mc2}\sum_{k=1}^\infty\frac{1}{k!}(-1)^{\frac{k(k-1)}{2}}l_k(\alpha^{\ot k})=0.
	\end{eqnarray}

By the definietion of $\{l_n\}_{n\geqslant 1}$,
\begin{align*}
&\frac{1}{2!}(-1)l_2(\alpha\ot \alpha)=-\frac{1}{2}[m,m]_G,\\
&\frac{1}{3!}(-1)l_3(\alpha^{\ot 3})\\
&=-\frac{1}{3!}\Big(l_3^r(m\ot \tau\ot \tau)+l^r_3(\tau\ot m\ot \tau)+l_3^r(\tau\ot \tau\ot m), l_3^l(m\ot \tau\ot \tau)+l^l_3(\tau\ot m\ot \tau)+l_3^l(\tau\ot \tau\ot m)\Big)\\
&=-\frac{1}{3!}\Big(3l_3^r(m\ot \tau\ot \tau),3l_3^l(m\ot \tau\ot \tau) \Big)\\
&=-\Big(s^{-1}m\circ(s\tau\ot s\tau)-\tau{\circ}(sm\circ(\id \ot s\tau)),s^{-1}m\circ(s\tau\ot s\tau)-\tau{\circ}(sm\circ(s\tau \ot \id))\Big)
\end{align*}
 Then Equation $(\ref{mc2})$ implies that
	\begin{eqnarray}
		\label{ass2} [m,m]_{G}&=&0\\
		\label{avr}s^{-1}m\circ((s\tau)^{\ot 2})&-&\tau\bar{\circ}(sm\circ(\id \ot s\tau))=0\\
		\label{avl}s^{-1}m\circ((s\tau)^{\ot 2})&-&\tau\bar{\circ}(sm\circ(s\tau\ot \id))=0.
	\end{eqnarray}
	Define $\mu=s^{-1}\circ m\circ s^{\ot 2}: A^{\ot 2}\rightarrow A, A=\tau\circ s: A\rightarrow A$. Then
	Equation (\ref{ass2}) is equivalent to $$\mu\circ(\id\ot \mu)=\mu\circ(\mu\ot \id).$$ That is, $\mu$ is associative.
	 Equations (\ref{avr})  (\ref{avl}) imply
	$$\mu\circ(A\ot A)=A\circ(\mu\circ(\id\ot A)),\mu\circ(A\ot A)=A\circ(\mu\circ(A\ot \id)),$$
	which means that $A$ is an averaging operator on associative algebra $(R,\mu)$.
	Conversely, let $(R,\mu,A)$ be an averaging algebra. Define $m: (sR)^{\ot 2}\rightarrow sR$ , $\tau: sR\rightarrow R$ to be $m(sa\ot sb)=s\mu(a\ot b)$, $\tau(sa)=A(a)$. Then $(m,\tau)$ is a Maurer-Cartan element in $\calc_{AvA}(R)$.
\end{proof}

\bigskip

\begin{prop}Let $(R,\mu,A)$ be an averaging algebra and $\alpha=(m,\tau)$ be the corresponding Maurer-Cartan element in $\calc_{AvA}(R)$.  Then the underground complex of $L_\infty$-algebra $({\calc_{AvA}}(R),\{l^\alpha\})$ is exactly the cochain complex $sC^\bullet_{AvA}(R)$ defined in  Section \ref{sec:cohomologyava}.
\end{prop}
\begin{proof}Identify $C^\bullet(R)$ with ${\calc_{A}}(R)$ and identify $C_{AvO}^\bullet$ with $\Hom(k,R)\oplus \Hom(sR,R)\oplus {\calc_{AvO}}(R)_r^{\geqslant 2}\oplus {\calc_{AvO}}(R)_l^{\geqslant 2}$. Let $sf\in \Hom((sR)^{\ot n},sR)$. Then $$l_1^\alpha(sf)=\sum_{k=0}^\infty\frac{1}{k!}(-1)^{\frac{k(k+1)}{2}+k}l_{k+1}(\alpha^{\ot k}\ot sf)=l_2(m\ot sf)+\frac{1}{n!}(-1)^{\frac{n(n-1)}{2}}l_{n+1}(\tau^{\ot n}\ot sf),$$
	where $l_2(m\ot sf)=[m,sf]_G$, it's just the differential $\delta$ of Hochschild cochain complex.  We have \begin{align*}\frac{1}{n!}(-1)^{\frac{n(n-1)}{2}}l_{n+1}(\tau^{\ot n}\ot sf)=&\frac{1}{n!}(-1)^{\frac{n(n-1)}{2}+n}(l_{n+2}^r(sf\ot \tau^{\ot n}),l_{n+1}^{l}(sf\ot \tau^{\ot n})).
	\end{align*}
	Then by definition, 	$$\frac{1}{n!}(-1)^{\frac{n(n-1)}{2}+n}l_{n+2}^r(sf\ot \tau^{\ot n})=(-1)^n(f\circ((s\tau)^{\ot n})-\tau\circ(sf\circ(\id\ot (s\tau)^{\ot n-1}))),$$
	which is exactly the same as $(-1)^n\Phi^n_r$ defined in Section \ref{sec:cohomologyav}. Similarly, we can see $l_{n+1}^l(sf\ot \tau^{\ot n})$ is just $(-1)^n\Phi^n_l$. In particular, for $sf\in \Hom(sR,sR)$, $l_2(sf\ot\tau)=l_2^r(sf\ot\tau)=l_2^l(sf\ot \tau)=-(f\circ s\tau-\tau\circ (sf))$, it is the same as $-\Phi^1$.
	
	\smallskip
	
	For $g\in {\calc_{AvO}}(R)^r$, we have
	\begin{align*}\sum_{k=0}^\infty\frac{1}{k!}(-1)^{\frac{k(k+1)}{2}+k}l_{k+1}(\alpha^{\ot k}\ot g)=&-\frac{1}{2}\big(l_3(m\ot \tau\ot g)+l_3(\tau\ot m\ot g)\big)\\
		=&	-l_3(m\ot \tau\ot g)\\
		=&-\Big(\big(-s^{-1}m_2\circ(s\tau\ot sg)+\tau\circ(m\circ(\id\ot sg))\big)\\
		&-\big(s^{-1}m\circ(sg\ot s\tau)-g\bar\circ(m\circ(\id\ot s\tau ))\big)\Big)\\
		=&s^{-1}m_2\circ(s\tau\ot sg)-\tau\circ(m\circ(\id\ot sg))\\
		&-g\bar\circ(m\circ(\id\ot \tau))+s^{-1}m\circ(sg\ot s\tau),
	\end{align*}
	and it is the same as $\partial_r$ defined in Section \ref{sec:cohomologyao}. For $g\in {\calc_{AvO}}(R)^l$, it is all the same. Especially, when $g\in \Hom(k,R)$, then by definition, we have
	
	\begin{align*}
		\sum_{k=0}^\infty\frac{1}{k!}(-1)^{\frac{k(k+1)}{2}+k}l_{k+1}(\alpha^{\ot k}\ot g)=&-\frac{1}{2}\big(l_3(m\ot \tau\ot g)+l_3(\tau\ot m\ot g)\big)\\
		=&	-l_3(m\ot \tau\ot g)\\
		=&-\Big(-s^{-1}m\circ(s\tau\circ sg)+\tau\circ(m\circ(\id\ot sg))\\
		&-(s^{-1}m\circ(sg\ot s\tau)-\tau\circ(m\circ(sg\ot \id)))\Big)
	\end{align*}
	And it is the same as the operator $\partial_0$ defined  in Section \ref{sec:cohomologyao}.
	
\end{proof}

\begin{prop}
	Let $V$ be a graded space. Then a Maurer-Cartan element $\alpha$ in ${\calc_{AvA}}(V)$ will induce a $L_\infty$-algebra structure on ${\calc_{AvO}}(V)$. In particular, for an averaging algebra $(R,\mu,A)$, the cochain complex $C_{AvO}(R)$ is a differential graded Lie-algebra.
\end{prop}

\begin{proof}Given a Maurer-Cartan element $\alpha$, $\alpha$ will induce $L_\infty$-algebra structure $\{l_n^\alpha\}_{n\geqslant 1}$ on ${\calc_{AvA}}(V)$. And notice that all operations $l_n^\alpha$ can be restricted to ${\calc_{AvO}}(V)$.
	Thus ${\calc_{AvO}}(V)$ forms a $L_\infty$-subalgebra of $({\calc_{AvA}}(V),\{l_n^\alpha\}_{n\geqslant 1})$.
	
	Let  $(R,\mu,A)$ be an averaging algebra and $(m,\tau)$ be the corresponding Maurer-Cartan element in $\calc_{AvA}(R)$. 	Identify $C_{AvO}(R)$ with ${\calc_{AvO}}(R)$.  Since $m=-s\circ\mu\circ (s^{-1})^{\ot 2}\in \Hom((sR)^{\ot 2},sR)$ in $\calc_A(R)$, the restriction of $l_n^\alpha$ on ${\calc_{AvO}}(R)$ is zero for $n\geqslant 3$. Thus $C_{AvO}(R)$ is just a differential graded Lie-algebra.
\end{proof}

\section{Homotopy averaging algebras}

In this subsection, we'll define homotopy averaging algebras.

Recall that an $A_\infty$-algebra structure is equivalent to a Maurer-Cartan element in the graded Lie-algebra $\overline{{\calc_{A}}}(V):=(\Hom(\overline{T^c(sV)},sV),[-,-]_G)$ where $\overline{T^c(sV)}=\bigoplus\limits_{n=1}^{\infty}(sV)^{\ot n}$.
 We can define homotopy averaging algebra structure in similar way.  For a graded vector space $V$, consider the following subspace of ${\calc_{AvA}}(V)$:
\[\overline{{\calc_{AvA}}}(V):=\overline{{\calc_{A}}}(V)\bigoplus\Hom(sV,V)\bigoplus{\calc_{AvO}}(V)_r^{\geqslant 2}\bigoplus {\calc_{AvO}}(V)_l^{\geqslant 2}.\]

Obviously, $\overline{{\calc_{AvA}}}(V)$ is  a $L_\infty$-subalgebra of $\overline{{\calc_{AvA}}}(V)$ with the operations $\{l_n\}_{n\geqslant 1}$ restricted on $\overline{{\calc_{AvA}}}(V)$.

Then we have the following definition:
\begin{defn}
	Let $V$ be a graded space. A homotopy averaging algebra( $Av_\infty$) structure on $V$ is a Maurer-Cartan element in the $L_\infty$-algebra $\overline{{\calc_{AvA}}}(V)$.
\end{defn}
Let's describe this structure explicitly.

A Maurer-Cartan element $\alpha$ in $\overline{{\calc_{AvA}}}(V)$ corresponds to a family of operators:$$\{b_n\}_{n\geqslant 1}\cup\{c\}\cup\{c_n^r\}_{n\geqslant 2}\cup\{c_n^l\}_{n\geqslant 2}$$ where $b_n:(sV)^{\ot n}\rightarrow sV$, $c: sV\rightarrow V$ belongs $\overline{\calc_{AvA}}(V)$, $\Hom(sV,V)$ respectively,  and  $c_n^r: (sV)^{\ot n}\rightarrow V, c_n^l:(sV)^{\ot n}\rightarrow V$ belongs to ${\calc_{AvO}}(V)(V)_r^{\geqslant 2}$, ${\calc_{AvO}}(V)^{\geqslant 2}_l$ respectively.  All these operators are  of degree $-1$. Then $\alpha$ satisfying the Maurer-Cartan equation implies that the family of operators satisfies the following equations:
\begin{align*}
	&\sum_{i+j=n}b_{n-j+1}\circ(\id^{\ot i}\ot b_j\ot \id^{\ot n-i-j})=0,\\
	&\sum_{m=1}^n\sum_{l_1+\dots+l_m=n}\Big(s^{-1}b_m\circ(sc_{l_1}^r\ot \dots \ot sc_{l_m}^r)-c_{l_1}^r\bar\circ(b_m\circ(\id\ot sc_{l_2}^r\ot \dots \ot sc_{l_m}^r))\Big)=0,\\
	&\sum_{m=1}^n\sum_{l_1+\dots+l_m=n}\Big(s^{-1}b_m\circ(sc_{l_1}^l\ot \dots \ot sc_{l_m}^l)-c_{l_m}^l\bar\circ(b_m\circ(sc_{l_1}^l\ot \dots \ot sc_{l_{m-1}}^l\ot \id))\Big)=0.
\end{align*}
for any $n\geqslant 1$, where $c_1^r=c_1^l=c$.

Then we define $m_n=s^{-1}\circ b_n\circ s^{\ot n}: V^{\ot
	n}\rightarrow V$,  $A=c\circ s: V\rightarrow V$ ,
$A_n^r=c_n^r\circ s^{\ot n}: V^{\ot n}\rightarrow V$ and
$A_n^l=c_n^l\circ s^{\ot n}: V^{\ot n}\rightarrow V$ for all
$n\geqslant 1$, where $|m_n|=n-2, |A|=0, |A_n^r|=|A_n^l|=n-1$.

These operators $\{m_n\}_{n\geqslant1}\cup \{A\}\cup\{A_n^r\}_{n\geqslant 2}\cup \{A_n^l\}_{n\geqslant 2}$ satisfies the following identities:
\smallskip
\begin{itemize}
	
	\item[(i)]\begin{align*}
		\sum_{i+j=n}(-1)^{i+jk}m_{n-j+1}\circ(\id^{\ot i}\ot m_j\ot \id^{\ot n-i-j})=0,
	\end{align*}
	
	\item[(ii)]\begin{align*}
		&\sum_{k=1}^n\sum_{l_1+\dots+l_k=n}(-1)^{\varepsilon}\Big\{m_k\circ(A_{l_1}^r\ot \dots \ot A_{l_k}^r)\\
		&-\sum_{p+q+1=l_1}(-1)^{q\cdot\sum\limits_{j=2}^k(l_j-1)+kp+q}A_{l_1}^r\circ(\id^{\ot p}\ot m_k\circ(\id\ot A_{l_2}^r\ot \dots \ot A_{l_k}^r)\ot \id^{\ot q})\Big\}=0,
	\end{align*}
	
	\item[(iii)]\begin{align*}&\sum_{k=1}^n\sum_{l_1+\dots+l_k=n}(-1)^{\varepsilon}\Big\{m_k\circ(A_{l_1}^l\ot \dots \ot A_{l_k}^l)\\
		&-\sum_{p+q+1=l_k}(-1)^{p\cdot\sum\limits_{j=1}^{k-1}(l_j-1)+(l_k-1)\cdot\sum\limits_{j=1}^{k-1}l_j+kp+q}A_{l_k}^l\circ(\id^{\ot p}\ot m_k\circ( A_{l_1}^l\ot \dots \ot A_{l_{k-1}}^l\ot \id)\ot \id^{\ot q})\Big\}=0
	\end{align*}
\end{itemize}
where$A_{1}^l=A_1^r=A$, and  $\varepsilon=\sum\limits_{j=1}^k\frac{l_j(l_j-1)}{2}+\frac{k(k-1)}{2}+\sum\limits_{j=1}^k(l_j-1)(l_1+\dots+l_{j-1})=\frac{n(n-1)}{2}+\frac{k(k-1)}{2}+\sum\limits_{j=1}^k(k-j+1)l_j$.

The equation $\mathrm{(i)}$ is just the definition of $A_\infty$-algebras.

\smallskip

Let's compute the equation $\mathrm{(ii) (iii)}$ for $n=1,2$:
	\begin{align*}
		(1)\  &n=1, -m_1\circ A+ A\circ m_1=0\\
		(2)\  &n=2,
		m_2\circ(A\ot A)-A\circ(m_2\circ(\id\ot A))=-m_1\circ A_{2}^r+A_2^r\circ(m_1\ot\id)+A_2^r\circ(\id\ot m_1),\\
		&\ \ \ \ \ \ \  \ \ \ \ m_2\circ(A\ot A)-A\circ(m_2\circ(A\ot \id))=-m_1\circ A_{2}^l+A_2^l\circ(m_1\ot\id)+A_2^l\circ(\id\ot m_1).
	\end{align*}
The equation $(1)$ above implies that $A$ is compatible with the differential, so operator $A$ can be induced on the homology of the complex $(V,m_1)$. The equation $(2)$ says that $A$ is an averaging operator up to homotopy with respect to $m_2$ on $V$ and the obstructions are $A_2^r$ and $A_2^l$.

So we get another definition of homotopy averaging algebra.
\begin{defn}Let $V$ be a graded space. Assume that $V$ is endowed with a family of operators $\{m_n: V^{\ot n}\rightarrow V\}_{n\geqslant 1}\cup\{A: V\rightarrow V\}\cup\{A_n^r:V^{\ot n}\rightarrow V\}_{n\geqslant 2}\cup\{A_n^l: V^{\ot n}\rightarrow V\}_{n\geqslant 2}$ with $|m_n|=n-2$, $|A|=0$, $|A_n^r|=|A_n^l|=n-1$. If these operators satisfy the equations $\mathrm{(i)\ (ii)\ (iii)}$ above, we call  $V$ a homotopy averaging algebra.
\end{defn}

\section*{Appendix: Proof of Theorem~\ref{averaging-L-infty}}

Firstly, let's display some facts we need to prove Theorem \ref{averaging-L-infty}.

\begin{lemma}{\label{permutation}}Let $n \geqslant 1, 1\leqslant i\leqslant n-1$. Then for any $\pi\in S_n$, there exists a unique triple $(\delta,\sigma,\tau)$ with $\sigma\in S(i,n-i),\delta\in S_i, \tau\in S_{n-i}$ such that $\pi(l)=\sigma\delta(l)$ for $1\leqslant l\leqslant i$, and $\pi(i+m)=\sigma(i+\tau(m))$ for $1\leqslant m\leqslant n-i$.
\end{lemma}

Let $V$ be a graded space. The operation $``\bar\circ"$ on $\Hom(T^c(V),V)$ satisfying the Pre-Jacobi identity:

\begin{lemma}
	For any homogeneous elements $f; g_1,\dots, g_m; h_1,\dots,h_n$ in $\Hom(T^c(V),V)$, the following identity holds:
	\begin{eqnarray*}
		&&\Big(f\bar{\circ}(g_1,\dots,g_m)\Big)\bar{\circ}\Big(h_1,\dots,h_n\Big)=\\
		&&\sum\limits_{0\leqslant i_1\leqslant i_2\leqslant \dots \leqslant i_m\leqslant n}(-1)^{\sum\limits_{k=1}^m|g_k|(\sum\limits_{j=1}^{i_k}|h_j|)}f\bar{\circ}\Big(h_1,\dots,h_{i_1},g_1\bar{\circ}(h_{i_1+1},\dots),\dots, h_{i_m}, g_m\bar\circ(h_{i_m+1},\dots)\dots\Big)
	\end{eqnarray*}
\end{lemma}

Now, let's prove the theorem \ref{averaging-L-infty}.

{\noindent{\bf{Theorem \ref{averaging-L-infty}.}} \it Let $V$ be a graded space. Then $\calc_{AvA}(V)$ endowed with operations $\{l_n\}_{n\geqslant 1}$ defined above forms a $L_\infty$-algebra.
}

\begin{proof}
	
	By the definition of $L_\infty$-algebras, we need to check the equation in $C_{AvA}(V)$ :
	\begin{eqnarray*}{\color{red}(\ref{def-L-infty})}\sum_{i+j=n}\sum_{\sigma\in \mathrm{S(i,n-i)}}(-1)^{i(n-i)}\chi(\sigma,x_1,\dots,x_n)l_{n-i+1}\Big(l_i(x_{\sigma(1)},\dots,x_{\sigma(i)}),x_{\sigma(i+1)},\dots, x_{\sigma(n)}\Big)=0.
	\end{eqnarray*}
	
	If all $x_i$ are contained in $\calc_{A}(V)=\Hom(T^c(sV),sV)$, then equation \ref{def-L-infty} is just the Jacobi identity on the graded Lie algebra $\calc_{A}(A)$.
	Apart from this, by the definition of $l_n$, the term $l_{n-i+1}\Big(l_i(x_{\sigma(1)},\dots,x_{\sigma(i)}),x_{\sigma(i+1)},\dots, x_{\sigma(n)}\Big) $ is trivial unless there are exactly two elements in $\{x_1,\dots,x_n\}$ coming from $\calc_{A(V)}$ and all other $n-2$ elements are contained in  $\calc_{AvO}(V)_l$ or $\calc_{AvO}(V)_r$. We assume that $x_1=sh_1\in \Hom((sV)^{\ot n_1},sV)$ and $x_2=sh_2\in \Hom((sV)^{\otimes n_2},sV)$. Then $n$ must be $n_1+n_2+1$. And we assume $x_i=g_{i-2}\in \calc_{AvO}(V)_r$. When $g_i\in \Hom(k,V)\oplus\Hom(sV,V)\oplus\Hom(T^c(sV)_\geqslant 2,V)_l$, the proof is all the same.
	
	For simplicity, we assume that all $g_i$ are not in $\Hom(s,V)$. If there is some $g_i\in \Hom(k,V)$, the calculation is similar.

	So now, we are going to check the equation \ref{def-L-infty} for
	$x_1=sh_1,x_2=sh_2,x_3=g_1,\dots, x_{n_1+n_2+1}=g_{n_1+n_2-1}$ with $g_i\in \Hom(sV,V)\oplus\Hom(k,V)\oplus\Hom(sV,V)\oplus\Hom(T^c(sV)_\geqslant 2,V)_r$. Then the term $l_{n-i+1}\Big(l_i(x_{\sigma(1)},\dots,x_{\sigma(i)}),x_{\sigma(i+1)}$ vanishes unless $i=2,n_1+1,n_2+1$. Let's compute them one by one.
	
	\medskip
	
	{\color{red}(A)} When $i=2$, the $(2,n-2)$-shuffle is identity map, we have
	\begin{align*}
		&l_{n_1+n_2-1}(l_2(sh_1,sh_2),g_1,\dots,g_{n_1+n_2-1})\\
		=&l_{n_1+n_2-1}([sh_1,sh_2]_G,g_1,\dots,g_{n_1+n_2-1})\\
		=&\sum_{\pi\in S_{n_1+n_2-1}}(-1)^{\varepsilon}\Big\{s^{-1}(sh_1\bar{\circ} sh_2)\circ(sg_{\pi(1)},\dots,sg_{\pi(n_1+n_2-1)})\\
		&-(-1)^{(|g_{\pi(1)}|+1)(|h_1|+|h_2|)}g_{\pi(1)}\bar{\circ}\Big((sh_1\bar{\circ} sh_2)\circ(\id_{sV}\ot sg_{\pi(2)}\ot \dots \ot sg_{\pi(n)})\Big)\Big\}\\
		&+\sum_{\pi\in {S_{n_1+n_2-1}}}(-1)^{\varepsilon}(-1)^{(|h_1|+1)(|h_2|+1)+1}\Big\{s^{-1}(sh_2\bar\circ sh_1)\circ(sg_{\pi(1)},\dots,sg_{\pi(n_1+n_2-1)})\\
		&\ \ \ \ \ \ \ -(-1)^{(|g_{\pi(1)}|+1)(h_1+h_2)}g_{\pi(1)}\bar\circ \big((sh_2\bar\circ sh_2)\circ(\id_{sV}\ot sg_{\pi(2)}\ot \dots\ot sg_{\pi(n_1+n_2-1)})\big)\Big\}\\
		=&\sum_{\pi\in S_{n_1+n_2-1}}(-1)^{\varepsilon}\Big\{\sum\limits_{i=0}^{n_1-1}(-1)^{(|h_2|+1)(\sum\limits_{j=1}^i(|g_{\pi(j)}|+1))}  h_1\circ(sg_{\pi(1)},\dots,sg_{\pi(i)},sh_2\circ(sg_{\pi(i+1)},\dots,sg_{\pi(i+n_2)}), \dots,sg_{\pi(n_1+n_2-1)})\Big\}\\
		&+\sum_{\pi\in S_{n_1+n_2-1}}(-1)^{\varepsilon}(-1)^{1+(|g_{\pi(1)}|+1)(|h_1|+|h_2|)} g_{\pi(1)}\bar{\circ}\Big\{sh_1\circ\Big(sh_2\circ(\id_{sV}\ot sg_{\pi(2)}\ot \dots sg_{\pi(n_2)})\ot \dots\ot sg_{\pi(n_1+n_2-1)}\Big)\\
		&+\sum\limits_{i=1}^{n_1-1}(-1)^{(|h_2|+1)(\sum\limits_{j=2}^i(|g_{\pi(j)}|+1))}sh_1\circ\Big(\id_{sV}\ot sg_{\pi(2)}\ot \dots \ot sg_{\pi(i)}\ot sh_2\circ(sg_{\pi(i+1)},\dots,sg_{\pi(i+n_2)})\ot \dots \ot sg_{\pi(n_1+n_2-1)}\Big)\Big\}\\
		&+\sum_{\pi\in S_{n_1+n_2-1}}(-1)^{\varepsilon}(-1)^{(|h_1|+1)(|h_2|+1)+1} \sum\limits_{i=0}^{n_2-1}(-1)^{(|h_1|+1)(\sum\limits_{j=1}^i(|g_{\pi(j)}|+1))}\bullet\\
		&\ \ \ \ \ \ \Big\{h_2\circ\Big(sg_{\pi(1)}\ot \dots \ot sg_{\pi(i)}\ot sh_1\circ(sg_{\pi(i+1)}\ot \dots\ot sg_{\pi(i+n_1)})\ot \dots \ot sg_{\pi(n_1+n_2-1)}\Big)\Big\}\\
		&+\sum_{\pi\in S_{n_1+n_2-1}}(-1)^{\varepsilon}(-1)^{(|h_1|+1)(|h_2|+1)+(|g_{\pi(1)}|+1)(|h_1|+|h_2|)}\bullet \\
		&\ \ \ \ \ \ g_{\pi(1)}\bar{\circ}\Big\{sh_2\circ\Big(sh_1\circ\big(\id_{sV}\ot sg_{\pi(2)}\ot \dots \ot sg_{\pi(n_2)}\big)\ot \dots \ot sg_{\pi(n_1+n_2-1)}\Big)\\
		&+\sum_{i=1}^{n_2-1}(-1)^{(|h_1|+1)(\sum\limits_{j=2}^i(|g_{\pi(j)}|+1))}sh_2\circ\Big(\id_{sV},sg_{\pi(2)},\dots,sg_{\pi(i)},sh_1\circ\big(sg_{\pi(i+1)},\dots,sg_{\pi(i+n_2)}\big),\dots sg_{\pi(n_1+n_2-1)}\Big)\Big\}.
	\end{align*}
	
	where
	$(-1)^\varepsilon=\chi(\pi,g_1,\dots,g_{n_1+n_2-1})(-1)^{(n_1
		+n_2-1)(h_1+h_2)+\sum\limits_{k=1}^{n_1+n_2-2}\sum\limits_{j
			=1}^k|g_{\pi(j)}|}$.
	\bigskip
	
	{\color{red}(B)}  When $i=n_2+1$,
	
	\begin{align*}
		&\sum_{\rho\in
			S(n_2+1,n_1)}(-1)^{(n_2+1)n_1}\chi(\rho,x_1,\dots,x_{n_1+n
			_2+1})l_{n_1+1}\big(l_{n_2+1}(x_{\rho(1)},\dots,x_{\rho(
			n_2+1)}),\dots,x_{\rho(n_1+n_2+1)}\big)\\
		=&\sum_{\sigma\in S(n_2,n_1-1)}(-1)^{(n_2+1)n_1}\chi(\sigma,g_1,\dots,g_{n_1+n_2-1})(-1)^{(|h_1|+1)(|h_2|+1+\sum\limits_{j=1}^{n_2}|g_{\sigma(j)}|)+n_2+1}\bullet \\
		&\ \ \ \ l_{n_1+1}\Big(l_{n_2+1}\big(sh_2,g_{\sigma(1)},\dots,g_{\sigma(n_2)}\big),sh_1,g_{\sigma(n_2+1)},\dots,g_{\sigma(n_1+n_2-1)}\Big)\\
		=& \sum_{\sigma\in S(n_2,n_1-1)}(-1)^{(n_2+1)n_1}\chi(\sigma,g_1,\dots,g_{n_1+n_2-1})(-1)^{(|h_1|+1)(|h_2|+1+\sum\limits_{j=1}^{n_2}|g_{\sigma(j)}|)+n_2+1}\bullet\\
		&(-1)^{(|h_1|+1)(|h_2|+1+\sum\limits_{j=1}^{n_2}|g_{\sigma(j)}|+n_2-1)+1} l_{n_1+1}\Big(sh_1,\underbrace{l_{n_2+1}(sh_2,g_{\sigma(1)},\dots,g_{\sigma(n_2)})}_{\color{red}:=\widehat{h_2}},g_{\sigma(n_2+1)},\dots,g_{\sigma(n_1+n_2-1)}\Big)\\
		=& \sum_{\sigma\in S(n_2,n_1-1)}\chi(\sigma,g_1,\dots,g_{n_1+n_2-1})(-1)^{(|h_1|+1)(n_2-1)+(n_2+1)n_1+n_2}l_{n_1+1}\Big(sh_1,\widehat{h_2},g_{\sigma(n_2+1)},\dots,g_{\sigma(n_2+n_1-1)}\Big)\\
		=&\sum_{\sigma\in S(n_2,n_1-1)}\sum_{\tau\in S_{n_1-1}}(-1)^{\theta_1}\Big\{\underbrace{h_1\circ\Big(s\widehat{h_2}\ot sg_{\sigma(n_2+\tau(1))}\ot \dots \ot sg_{\sigma(n_2+\tau(n_1-1))}\Big)}_{\color{red}B_1}\\
		&-(-1)^{(|\widehat{h_2}|+1)(|h_1+1|)}\underbrace{\widehat{h_2}\bar\circ\Big(sh_1\circ\big(\id_{sV}\ot sg_{\sigma(n_2+\tau(1))}\ot \dots\ot sg_{\sigma(n_2+\tau(n_1-1))}\big)\Big)}_{\color{red}B_2}\Big\}\\
		&+\sum_{\sigma\in S(n_2,n_1-1)}\sum_{\tau\in S(n_1-1)}\sum_{i=1}^{n_1-1}(-1)^{\theta_{2}} \Big\{\underbrace{h_1\circ\Big(sg_{\sigma(n_2+\tau(1))}\ot \dots \ot sg_{\sigma(n_2+\tau(i))}\ot s\widehat{h_2}\ot \dots \ot sg_{\sigma(n_2+\tau(n_1-1))}\Big)}_{\color{red}B_3}\\
		&-(-1)^{(|h_1|+1)(|g_{\sigma(n_2+\tau(1))}|+1)}\underbrace{g_{\sigma(n_2+\tau(1))}\bar\circ\Big(sh_1\circ\big(\id_{sV}\ot sg_{\sigma(n_2+\tau(2))}\ot \dots \ot sg_{\sigma(n_2+\tau(i))}\ot s\widehat{h_2}\ot \dots \ot sg_{\sigma(n_2+\tau(n_1-1))}\big)\Big)}_{\color{red}B_4}\Big\}
	\end{align*}
	
	where \begin{align*}(-1)^{\theta_1}=&\chi(\sigma,g_1,\dots,g_{n_1+n_2-1})\chi(\tau,g_{\sigma(n_2+1)},\dots,g_{\sigma(n_2+n_1-1)})\\
		&\cdot(-1)^{(|h_1|+1)(n_2-1)+(n_2+1)n_1+n_2}(-1)^{n_1(|h_1|+1)+(n_1-1)(|\widehat{h_2}|)+\sum\limits_{k=1}^{n_1-2}\sum\limits_{j=1}^k|g_{\sigma{(n_2+\tau(j))}}|}.\\
		(-1)^{\theta_2}=&\chi(\sigma,g_1,\dots,g_{n_1+n_2-1})\chi(\tau,g_{\sigma(n_2+1)},\dots,g_{\sigma(n_2+n_1-1)})\cdot(-1)^{(|h_1|+1)(n_2-1)+(n_2+1)n_1+n_2}\\
		&\cdot(-1)^{n_1(|h_1|+1)+\sum\limits_{k=1}^{n_1-2}\sum\limits_{j=1}^k|g_{\sigma(n_2+\tau(j))}|+\sum\limits_{j=1}^i|g_{\sigma(n_2+\tau(j))}|+(n_1-i-1)|\widehat{h_2}|+(|\widehat{h_2}|)(\sum\limits_{j=1}^i|g_{\sigma(n_2+\tau(j))}|)+i}
	\end{align*}
	
	\bigskip
	
	Let's compute the terms $B_1,B_2,B_3,B_4$ together with their coefficients.	
	
	\begin{align*}
		{\color{red}B_1}=&\sum_{\sigma\in S(n_2,n_1-1)}\sum_{\tau\in S_{n_1-1}}(-1)^{\theta_1}h_1\circ(s\widehat{h_2}\ot \dots\ot sg_{\sigma(n_2+\tau(1))}\ot \dots\ot sg_{\sigma(n_2+\tau(n_1-1))})\\
		=&\sum_{\sigma\in S(n_2,n_1-1)}\sum_{\tau\in S_{n_1-1}}(-1)^{\theta_1}h_1\circ\Big(sl_{n_2+1}(sh_2,g_{\sigma(1)},\dots,g_{\sigma(n_2)}),sg_{\sigma(n_2+\tau(1))}\ot \dots \ot sg_{\sigma(n_2+\tau(n_1-1))}\Big)\\
		=&\sum_{\sigma\in S(n_2,n_1-1)}\sum_{\tau\in S_{n_1-1}}(-1)^{\theta_1}\sum_{\delta\in S_{n_2}}\chi(\delta,g_{\sigma(1)},\dots, g_{\sigma(n_2)})(-1)^{n_2(|h_2|+1)+\sum\limits_{k=1}^{n_2-1}\sum\limits_{j=1}^k|g_{\sigma(\delta(j))}|}\bullet\\
		&\ \ \  h_1\circ\Big\{\Big(sh_2\circ(sg_{\sigma\delta(1)},\dots,sg_{\sigma\delta(n_2)})\ot  sg_{\sigma(n_2+\tau(1))}\ot \dots \ot sg_{\sigma(n_2+\tau(n_1-1))}\Big)-(-1)^{(|g_{\sigma\delta(1)}|+1)(|h_2|+1)}\bullet\\
		&\ \ \ \ \Big(sg_{\sigma\delta(1)}\bar\circ\big(sh_2\circ (\id_{sV}\ot sg_{\sigma\delta(2)}\ot \dots\ot sg_{\sigma\delta(n_2)})\big)\ot sg_{\sigma(n_2+\tau(1))}\ot \dots \ot sg_{\sigma(n_2+\tau(n_1-1))}\Big)\Big\}\\
		=&\sum_{\pi\in S(n_1+n_2-1)}(-1)^{\eta_1}h_1\circ\Big\{\Big(sh_2\circ\big(sg_{\pi(1)}\ot \dots \ot sg_{\pi(n_2)}\big)\ot sg_{\pi(n_2+1)}\ot \dots \ot sg_{\pi(n_2+n_1-1)}\Big)\\
		&-(-1)^{(|g_{\pi(1)}|+1)(|h_2|+1)}\Big(sg_{\pi(1)}\bar\circ\big(sh_2\circ(\id_{sV}\ot sg_{\pi(2)}\ot \dots\ot sg_{\pi(n_2)})\big)\ot sg_{\pi(n_2+1)}\ot \dots \ot sg_{\pi(n_2+n_1-1)}\Big)\Big\}
	\end{align*}
	where
	\begin{align*}(-1)^{\eta_1}=&(-1)^{\theta_1}\chi(\delta,g_{\sigma(1)},\dots,g_{\sigma{n_2}})(-1)^{n_2(|h_2|+1)+\sum\limits_{k=1}^{n_2-1}\sum\limits_{j=1}^k|g_{\sigma(\delta(j))}|}\\
		=&\chi(\pi,g_1,\dots,g_{n_1+n_2-1})(-1)^{(|h_1|+|h_2|)(n_1+n_2-1)+1+\sum\limits_{k=1}^{n_1+n_2-2}\sum\limits_{j=1}^k|g_{\pi(j)}|}.
	\end{align*}
	
	Notice that, in the last equality above, we use lemma \ref{permutation} to replace the triple $(\delta,\sigma, \tau)$ by its corresponding permutation $\pi\in S_{n_1+n_2-1}$ and we use the fact
	$$\chi(\pi,g_1,\dots,g_{n_1+n_2-1})=\chi(\sigma,g_1,\dots,g_{n_1+
		n_2-1})\chi(\tau,g_{\sigma(n_2+1)},\dots,g_{\sigma(n_2+n_1-1)})\chi(
	\delta,g_{\sigma(1)},\dots,g_{\sigma(n_2)}).$$	\medskip
	
	Similarly, we have
	\begin{align*}
		{\color{red}B_2}=&\sum_{\sigma\in S(n_2,n_1-1)}\sum_{\tau\in S_{n_1-1}}(-1)^{\theta_1}(-1)^{1+(|\widehat{h_2}|+1)(|h_1|+1)}
		\widehat{h_2}\bar\circ\Big(sh_1\circ(\id_{sV}\ot sg_{\sigma(n_2+\tau(1))}\ot \dots \ot sg_{\sigma(n_2+\tau(n_1-1))})\Big)\\
		=&\sum_{\pi\in S(n_1+n_2-1)}(-1)^{\eta_2}\Big\{h_2\circ\Big(sg_{\pi(1)}\ot \dots\ot sg_{\pi(n_2)}\Big)-(-1)^{(|h_2|+1)(|g_{\pi(1)}|+1)}\bullet\\ &\Big(g_{\pi(1)}\bar\circ\big(sh_2\circ(\id_{sV}\ot sg_{\pi(2)}\ot \dots \ot sg_{\pi(n_2)})\big)\Big)\Big\}\bar\circ\underbrace{\Big(sh_1\circ\big(\id_{sV}\ot sg_{\pi(n_2+1)}\ot \dots \ot sg_{\pi(n_2+n_1-1)}\big)\Big)}_{\color{red}:=\widetilde{h_1}}\\
		=&\sum_{\pi\in S(n_1+n_2-1)}(-1)^{\eta_2}\Big\{\sum_{k=1}^{n_2}(-1)^{(\sum\limits_{j=k+1}^{n_2}(|g_{\pi(j)}|+1))(|\widetilde{h_1}|)}h_2\circ\Big(sg_{\pi(1)}\ot \dots\ot sg_{\pi(k)}\bar\circ\widetilde{h_1}\ot sg_{\pi(k+1)}\ot \dots \ot sg_{\pi(n_2)}\Big)\Big\}\\
		&\\
		&+(-1)^{\eta_2}(-1)^{1+(|h_2|+1)(|g_{\pi(1)}|+1)}\Big\{g_{\pi(1)}\bar\circ\Big(sh_2\circ(\id_{sV}\ot sg_{\pi(2)}\ot \dots\ot sg_{\pi(n_2)}),\widetilde{h_1}\Big)\\
		&+(-1)^{|\widetilde{h_1}|\big(\sum\limits_{j=2}^{n_2}(|g_{\pi(j)}|+1)\big)}g_{\pi(1)}\bar\circ\Big(sh_2\circ\big(\widetilde{h_1}\ot sg_{\pi(2)}\ot \dots \ot sg_{\pi(n_2)}\big)\Big)\\
		&+\sum_{k=2}^{n_2}(-1)^{|\widetilde{h_1}|\big(\sum\limits_{j=k+1}^{n_2}(|g_{\pi(j)}|+1)\big)}g_{\pi(1)}\bar\circ\Big(sh_2\circ\big(\id_{sV}\ot sg_{\pi(2)}\ot \dots \ot sg_{\pi(k)}\bar\circ\widetilde{h_1}\ot sg_{\pi(k+1)}\ot \dots \\
		&\dots \ot sg_{\pi(n_2)}\big)\Big)+(-1)^{|\widetilde{h_1}|\big(\sum\limits_{j=2}^{n_2}(|g_{\pi(j)}|+1)+|h_2|+1\big)}g_{\pi(1)}\bar\circ\Big(\widetilde{h_1},sh_2\circ(\id_{sV}\ot sg_{\pi(2)}\ot \dots\ot sg_{\pi(n_2)})\Big)\Big\}
	\end{align*}
	where \begin{align*}(-1)^{\eta_2}=&(-1)^{\eta_1}(-1)^{1+(|\widehat{h_2}|+1)(|h_1|+1)}\\
		=&\chi(\pi,g_1,\dots,g_{n_1+n_2-1})(-1)^{(|h_1|+|h_2|)(n_1+n_2-1)+1+\sum\limits_{k=1}^{n_1+n_2-2}\sum\limits_{j=1}^k|g_{\pi(j)}|+1+(|\widehat{h_2}|+1)(|h_1|+1)}\\
		=&\chi(\pi,g_{1},\dots,g_{n_1+n_2-1})(-1)^{(|h_1|+|h_2|)(n_1+n_2-1)+\sum\limits_{k=1}^{n_1+n_2-2}\sum\limits_{j=1}^k|g_{\pi(j)}|+\big(|h_2|+n_2-1+\sum\limits_{j=1}^{n_2}|g_{\pi(j)}|\big)\big(|h_1|+1\big)}
	\end{align*}
	
	\begin{align*}
		{\color{red} B_3}=&\sum_{\sigma\in S(n_2,n_1-1)}\sum_{\tau\in S(n_1-1)}\sum_{i=1}^{n_1-1}(-1)^{\theta_2}h_1\circ\Big\{sg_{\sigma(n_2+\tau(1))}\ot \dots \ot sg_{\sigma{n_2+\tau(i)}}\ot s\widehat{h_2}\ot \dots \ot sg_{\sigma(n_2+\tau(n_1-1))}\Big\}\\
		=&\sum_{\pi\in S(n_2+n_1-1)}\sum_{i=1}^{n_1-1}(-1)^{\eta_3}h_1\circ\Big\{sg_{\pi(n_2+1)}\ot \dots \ot sg_{\pi(n_2+i)}\ot sh_2\circ(sg_{\pi(1)}\ot \dots \ot sg_{\pi(n_2)})\ot \dots \ot sg_{\pi(n_1+n_2-1)}\\
		&	-(-1)^{(|h_2|+1)(|g_{\pi(1)}|+1)}sg_{\pi(n_2+1)}\ot \dots \ot sg_{\pi(n_2+i)}\ot sg_{\pi(1)}\bar\circ\Big(sh_2\circ(\id_{sV}\ot sg_{\pi(2)}\ot \dots \ot sg_{\pi(n_2))}\Big)\ot \dots\\
		&\dots \ot sg_{\pi(n_1+n_2-1)}\Big\}
	\end{align*}
	where \begin{align*}(-1)^{\eta_3}=&(-1)^{\theta_2}\chi(\delta,g_{\sigma(1)},\dots,g_{\sigma(n_2)})(-1)^{n_2(|h_2|+1)+\sum\limits_{k=1}^{n_2-1}|g_{\pi(j)}|}\\
		=&\chi(\sigma,g_1,\dots,g_{n_1+n_2-1})\chi(\tau,g_{\sigma(n_2+1)},\dots,g_{\sigma(n_2+n_1-1)})\cdot(-1)^{(|h_1|+1)(n_2-1)+(n_2+1)n_1+n_2}\\
		&\cdot(-1)^{n_1(|h_1|+1)+\sum\limits_{k=1}^{n_1-2}\sum\limits_{j=1}^k|g_{(\sigma+\tau(j))}|+\sum\limits_{j=1}^i|g_{\sigma(n_2+\tau(j))}|+(n_1-i-1)|\widehat{h_2}|+|\widehat{h_2}|\big(\sum\limits_{j=n_2+1}^{n_2+i}|g_{\pi(j)}|\big)+i}\\
		&\cdot(-1)^{n_2(|h_2|+1)+\sum\limits_{k=1}^{n_2-1}|g_{\pi(j)}|}\chi(\delta,g_{\sigma(1)},\dots,g_{\sigma(n_2)})\\
		=&\chi(\pi,g_1,\dots,g_{n_1+n_2-1})(-1)^{(|h_1|+|h_2|)(n_1+n_2-1)+1+\sum\limits_{k=1}^{n_1+n_2-2}\sum\limits_{j=1}^k|g_{\pi(j)}|+(1+|h_2|+n_2+\sum\limits_{j=1}^{n_2}|g_{\pi(j)}|)(i+\sum\limits_{j=1}^i|g_{\pi(n_2+j)}|)}
	\end{align*}
	
	\begin{align*}
		{\color{red}B_4}=&\sum_{\pi\in S(n_2+n_1-1)}\sum_{i=1}^{n_1-1}(-1)^{\eta_4}g_{\pi(n_2+1)}\bar\circ\Big\{sh_1\circ\Big(\id_{sV}\ot sg_{\pi(n_2+2)}\ot \dots \ot sg_{\pi(n_2+i)}\ot sh_2\circ(sg_{\pi(1)}\ot \dots sg_{\pi(n_2)})\ot \dots\\
		& \ot sg_{\pi(n_2+n_1-1)}\Big)-(-1)^{(|h_2|+1)(|g_{\pi(1)}|+1)}sh_1\circ\Big(\id_{sV}\ot sg_{\pi(n_2+2)}\ot \dots \ot sg_{\pi(n_2+i)}\ot sg_{\pi(1)}\bar\circ\big(sh_2\circ(\i_{sV}\ot sg_{\pi(2)}\ot\dots\\
		&\dots sg_{\pi(n_2)})\big)\ot \dots\ot sg_{\pi(n_2+n_1-1)}\Big)\Big\}
	\end{align*}
	where $$(-1)^{\eta_4}=(-1)^{\eta_3}(-1)^{1+(|h_1|+1)(|g_{\pi(n_2+1)}|+1)}.$$
	
	\medskip
	
	{\color{red}(C)} When $i=n_1+1$,

	\begin{align*}
		&\sum_{\rho\in
			S(n_1+1,n_2)}(-1)^{(n_1+1)n_2}\chi(\rho,x_1,\dots,x_{n_1+n
			_2+1})l_{n_2+1}\big(l_{n_1+1}(x_{\rho(1)},\dots,x_{\rho(
			n_2+1)}),\dots,x_{\rho(n_1+n_2+1)}\big)\\
		=&\sum_{\sigma\in S(n_1,n_2-1)}(-1)^{(n_1+1)n_2}(-1)^{n_1+(|h_2|+1)(\sum\limits_{j=1}^{n_1}g_{\sigma(j)})}\chi(\sigma,g_1,g_2,\dots,g_{n_1+n_2-1})\bullet \\
		&l_{n_2+1}(l_1(sh_1,g_{\sigma(1)},\dots,g_{\sigma(n_1)}),sh_2,g_{\sigma(n_1+1)},\dots,g_{\sigma(n_1+n_2-1)})\\
		=&\sum_{\sigma\in S(n_1,n_2-1)}(-1)^{(n_1+1)n_2}(-1)^{n_1+(|h_2|+1)(\sum\limits_{j=1}^{n_1}g_{\sigma(j)})}(-1)^{(|h_2|+1)(\sum\limits_{j=1}^{n_1}g_{\sigma(j)}+|h_1|+n_1)+1}\chi(\sigma,x_1,\dots,x_{n_1+n_2+1})\bullet\\
		&	\ \ \ \ \ \ \ l_{n_2+1}(sh_2, l_{n_1+1}(sh_1,g_{\sigma(1)},\dots,g_{\sigma}(n_1)),g_{\sigma(n_1+1)},\dots,g_{\sigma(n_2+n_1-1)})\\
		=&\sum_{\sigma\in S(n_1,n_2-1)}(-1)^{(n_1+1)(n_2+1)+(|h_2|+1)(|h_1|+n_1)}\chi(\sigma,g_{1},\dots,g_{n_1})\bullet\\
		&\ \ \ \ \ \ \ l_{n_2+1}(sh_2, l_{n_1+1}(sh_1,g_{\sigma(1)},\dots,g_{\sigma}(n_1)),g_{\sigma(n_1+1)},\dots,g_{\sigma(n_2+n_1-1)})
	\end{align*}
	
	Switch the roles of $h_1,n_1$ and $h_2,n_2$ in the expansion of $$l_{n_1+1}\Big(sh_1,l_{n_2+1}(sh_2,g_{\sigma(1)},\dots,g_{\sigma(n_2)}),g_{\sigma(n_2+1)},\dots,g_{\sigma(n_1+n_2-1)}\Big),$$
	
	which has been computed in case {(B)}, then we will get the expansion of
	$$l_{n_2+1}(sh_2, l_{n_1+1}(sh_1,g_{\sigma(1)},\dots,g_{\sigma}(n_1)),g_{\sigma(n_1+1)},\dots,g_{\sigma(n_2+n_1-1)}).$$
	
	Then the expansion for case {(C)} comes out. Then, we can found that the terms of the same form appear exactly twice in the expansion of {{\color{red}(A)+(B)+(C)}}, and they have opposite signs. Thus $\mathrm {\color{red}(A)+(B)+(C)}=0$ holds. So $C_{AvA}(V)$ forms a $L_\infty$-algebra.
\end{proof}


\begin{thebibliography}{99}
	\bibitem{Agu} M.~Aguiar, \emph{Pre-Poisson algebras}, Lett. Math. Phys. \textbf{54} (2000) 263-277.



\bibitem{Bax} G.~Baxter, \emph{An analytic problem whose solution follows from a simple algebraic identity,}
Pacific J. Math. \textbf{10} (1960) 731-742.


\bibitem{Bir} G.~Birkhoff, \emph{Moyennes de fonctions born\'ees}, Coil. Internat. Centre Nat. Recherche
Sci. (Paris), Alg\`ebre Th\'eorie Nombres \textbf{24} (1949) 149-153.





\bibitem{Bong} N.~H.~Bong, \emph{Some apparent connection between baxter and averaging operators,}
J. Math. Anal. Appl. \textbf{56} (1976) 330-345.


\bibitem{Cao} W.~Cao, \emph{An algebraic study of averaging operators}, Ph.D. thesis, Rutgers University
at Newark (2000).

 \bibitem{CGX} Z.~Chen, H.~Ge and M.~Xiang,  \emph{Talk at The Third Conference of Operad Theory and Related Topics}, September 2020.

 \bibitem{KdF} J.~Kamp\'e de F\'eriet, \emph{Introduction to the statistical theory of turbulence, correlation
 and spectrum}, in The Institute of Fluid Dynamics and Applied Mathematics, Lecture
Series No. \textbf{8}, prepared by S.~I.~Pai (University of Maryland, 1950-51).

\bibitem{GamMil} J.~L.~B.~Gamlen and J.~B.~Miller, \emph{ Averaging and Reynolds operators on Banach
algebras II. spectral properties of averaging operators,} J. Math. Anal. Appl. \textbf{23} (1968)
183-197.

 \bibitem{GaoZhang}  X.~Gao and  T.~Zhang, \emph{Averaging algebras, rewriting systems
and Gr\"{o}bner每Shirshov bases},   J. Algebra Appl. \textbf{17} (2018), no. 7, 1850130, 26 pp.





\bibitem{Ger63} M. Gerstenhaber,   \emph{The cohomology structure
of an associative ring}. Ann. of Math. (2) \textbf{78} 1963
267-288.

\bibitem{Guo} L.~Guo, \emph{An Introduction to Rota每Baxter Algebra} (International Press (US) and Higher
Education Press (China), 2012).



\bibitem{KdF}  J.~Kamp\'e de F\'eriet, \emph{L'\'etat actuel du probl\`eme de la turbulence (I and II)}, La Sci.
A\'erienne \textbf{3} (1934) 9-34, \textbf{4} (1935) 12-52.

\bibitem{Kelley} J.~L.~Kelley, \emph{Averaging operators on $C^\infty (X)$}, Illinois J. Math. \textbf{2} (1958) 214-223.


\bibitem{LM02}T.~Lada and  M.~Markl, \emph{Strongly homotopy Lie algebras}, Communications in algebra, \textbf{23} (1995), no. 6, 2147-2161.

\bibitem{LS93} T.~Lada and J.~Stasheff, \emph{ Introduction to sh Lie algebras for physicists}, Internat. J. Theoret. Phys. \textbf{32} (1993), 1087-1103.

\bibitem{Loday} J.~L.~Loday, \emph{Dialgebras}, in Dialgebras and related operads, Lecture Notes in Math.
1763 (2002) 7-66.

\bibitem{Miller} J.~B.~Miller, \emph{Averaging and Reynolds operators on Banach algebra I, Representation
by derivation and antiderivations,} J. Math. Anal. Appl. \textbf{14} (1966) 527-548.

\bibitem{Moy} S.~T.~C.~Moy, \emph{Characterizations of conditional expectation as a transformation on
function spaces,} Pacific J. Math. \textbf{4} (1954) 47-63.



\bibitem{PeiGuo} J.~Pei and L.~Guo, \emph{Averaging algebras, Schr\"{o}der numbers, rooted trees and operads,}
J. Algebra Comb. \textbf{42} (2015) 73-109.

\bibitem{Rey} O.~Reynolds, \emph{On the dynamic theory of incompressible viscous fluids and the determination
of the criterion}, Phil. Trans. Roy. Soc. A \textbf{136} (1895) 123-164.


\bibitem{Rota} G.~C.~Rota, \emph{Reynolds operators}, Proceedings of Symposia in Applied Mathematics,
Vol. XVI (American Mathematical Society, Providence, R.I., 1964), pp. 70-83.

\bibitem{Rota} G.~C.~Rota, \emph{Baxter algebras and combinatorial identities I, II}, Bull. Amer. Math.
Soc. \textbf{75} (1969) 325-329, 330-334.

 \bibitem{STZ} Y.~Sheng, R.~Tang and C.~Zhu, \emph{The controling $L_\infty$-algebras, cohomology and homotopy of embedding tensors and Lie-Leibniz triples},  arXiv:2009.11096.
 
\bibitem{JS90} J.~Stasheff, \emph{Differential graded Lie algebras, quasi-Hopf algebras and higher homotopy algebras},  Quantum groups (Leningrad, 1990), 120-137, Lecture Notes in Math 1510, Springer, Berlin, 1992.
    
    
 \bibitem{Tri} A.~Triki, \emph{Extensions of positive projections and averaging operators}, J. Math. Anal.
 \textbf{153} (1990) 486每496.


	
	
\end{thebibliography}
\end{document}